\documentclass[10pt]{article}

\usepackage{amsthm,amssymb}
\usepackage{soul,color}
\usepackage{fullpage}
\usepackage{array}
\usepackage{enumerate}
\usepackage{dsfont}

\usepackage{epsfig,rotate}
\usepackage[]{amsmath}
\usepackage[]{amsfonts}
\usepackage[]{graphicx}

\numberwithin{equation}{section}

\newcommand{\be}{\begin{equation}}
\newcommand{\ee}{\end{equation}}
\newcommand{\ba}{\begin{array}}
\newcommand{\ea}{\end{array}}
\newcommand{\bea}{\begin{eqnarray*}}
\newcommand{\eea}{\end{eqnarray*}}
\newcommand{\bean}{\begin{eqnarray}}
\newcommand{\eean}{\end{eqnarray}}
      
\newtheorem{theorem}{Theorem}
 
\newtheorem{proposition}{Proposition}
\newtheorem{lemma}{Lemma}
\newtheorem{corollary}{Corollary}
\newtheorem{remark}{Remark}

\def\Re{{\text{\rm Re}}}
\def\Ima{{\text{\rm Im}}}
\def\Sess{\sigma_{\text{\rm ess}}}

\def\a{{\alpha}}

\def\d{{\delta}}
\def\e{{\varepsilon}}

\def\vf{{\varphi}}

\def\R{{\mathbf R}}
\def\Id{{\text{\rm Id}}}
\def\Ker{{\text{\rm Ker}}}
\def\div{{\text{\rm div}}}
\def\ds{{\displaystyle}}

\newcommand{\finproof}
{%
\mbox{}%
\nolinebreak%
\hfill%
\rule{2mm}{2mm}%
\medbreak%
\par%
}

\title{The plasmonic resonances of a bowtie antenna}
\author{
Eric Bonnetier 
\thanks{Institut Fourier, 
Universit\'e Grenoble-Alpes, 
BP 74, 38402 Saint-Martin-d'H\`eres Cedex, France,
({\tt Eric.Bonnetier@univ-grenoble-alpes.fr}).}
\and
Charles Dapogny 
\thanks{Laboratoire Jean Kuntzmann, 
Universit\'e Grenoble-Alpes, 
700 Avenue Centrale,
38401 Domaine Universitaire de Saint-Martin-d'H\`eres, France,
({\tt Charles.Dapogny@univ-grenoble-alpes.fr}).}
\and
Faouzi Triki
\thanks{
Laboratoire Jean Kuntzmann, 
Universit\'e Grenoble-Alpes, 
700 Avenue Centrale,
38401 Domaine Universitaire de Saint-Martin-d'H\`eres, France,
({\tt Faouzi.Triki@univ-grenoble-alpes.fr}).}
\and
Hai Zhang
\thanks{Department of Mathematics, 
Hong Kong University of Science and Technology, 
Clear Water Bay, Kowloon, HK,
({\tt haizhang@ust.hk}). }
}

\begin{document}

\maketitle

\begin{abstract}
\noindent Metallic bowtie-shaped nanostructures are very interesting objects in optics,
due to their capability of localizing and enhancing electromagnetic fields
in the vicinity of their central neck.
In this article, we investigate the electrostatic plasmonic resonances of two-dimensional bowtie-shaped domains by looking at the spectrum of their 
Poincar\'e variational operator. In particular, we show that the latter only consists of essential spectrum and fills the whole interval $[0,1]$. 
This behavior is very different from what occurs in the counterpart situation of a bowtie domain with only close-to-touching wings,
a case where the essential spectrum of the Poincar\'e variational operator is reduced to an interval strictly contained in $[0,1]$. 
We provide an explanation for this difference by showing that the spectrum of
the Poincar\'e variational operator of 
bowtie-shaped domains with close-to-touching wings has eigenvalues which densify and eventually fill the remaining intervals as the distance between the two wings tends to zero.
\end{abstract}

\section{Introduction}

Surface plasmons are strongly localized electromagnetic fields
that result from electron oscillations on the surface of metallic particles.
Typically, this resonant behavior occurs when the real parts of 
the dielectric coefficients of the particles are negative, and when
their size is comparable to or smaller than the wavelength of the excitation.
For instance, this is the case of gold or silver nanoparticles,
20-50 nm in diameter, 
when they are illuminated in the frequency range of visible light.
\medskip

The ability to confine, enhance and control electromagnetic fields in regions
of space smaller than or of the order of the excitation wavelength has stirred
considerable interest in surface plasmons over the last decade, as it opens the door
to a large number of applications in the domains of nanophysics, 
near-field microscopy, bio-sensing, nanolithography, and quantum computing, 
to name a few.
\medskip

A great deal of the mathematical work about plasmons has focused on the so-called 
electrostatic case, where the Maxwell system is reduced to a Helmholtz 
equation, and in the asymptotic limit when the particle diameter is small
compared to the frequency $\omega$ of the incident wave.
After proper rescaling, the study amounts to that of a conduction equation of the form
\begin{eqnarray}~\label{eq_cond1}
\textrm{div}\left(\e(\omega)^{-1}(x) \nabla u(x)) \right) &=& 0,
\end{eqnarray}
complemented with appropriate boundary or radiation conditions; see \cite{Hai_1,Hai_2} for a mathematical justification.
The electric permittivity $\e(\omega)$ in (\ref{eq_cond1}) takes different forms in the dielectric ambient medium, and inside the particle; in the latter situation, it
is usually modeled by a Drude-Lorentz law of the form:
\begin{eqnarray*}
\e(\omega) &=& \e_0\left( 1 -  \ds\frac{\omega_p^2}
{\omega^2 + i \omega \gamma} \right).
\end{eqnarray*}
where $\e_0$ is the electric permittivity of the vacuum, and where
$\omega_p$ and $\gamma$ respectively denote the plasma frequency and the conductivity
of the medium; see~\cite{MayergoyzFredkinZhang, Mayergoyz, Grieser, Hai_1, Hai_2, Hai_3}.
In the case of metals such as gold and silver, experimental data show that, for frequencies in the range $200-700$ $\mu$m, $\Re(\e(\omega)) < 0$, 
while the rate $\Ima(\e(\omega))$ of dissipation of electrostatic energy is small.
In this context, (\ref{eq_cond1}) gets close to a two-phase conduction equation with sign-changing coefficients, and it loses its elliptic character.
\medskip

In the above electrostatic approximation, the plasmonic resonances of a particle $D$ embedded in a homogeneous medium of permittivity 
$\e_0$ are precisely associated with values of the permittivity $\e$ inside the particle for which~(\ref{eq_cond1}) 
ceases to be well-posed. If the shape of the particle is sufficiently 
smooth, one may represent the solution $u$ to (\ref{eq_cond1}) via layer potentials, and then
characterize plasmon resonances as values of the contrast 
$\frac{\e + \e_0}{2(\e - \e_0)}$ which are eigenvalues of the associated 
Neumann-Poincar\'e integral operator 
${\mathcal K}^*_D$; see \cite{Mayergoyz, Hai_1}.
\medskip

Due to their key role in various physical contexts, the spectral properties of the Neumann-Poincar\'e operator 
have been the focus of numerous investigations~\cite{AndoKang,AmmariCiraoloKangLeeMilton,BDT,BonnetierTriki,BonnetierTriki_2}.
When the inclusion $D$ is smooth (say with ${\mathcal C}^{1,\a}$ boundary), 
${\mathcal K}^*_D$ is a compact operator, and so its spectrum $\sigma({\mathcal K}_D^*)$ consists in a sequence of
eigenvalues that accumulates to $0$~\cite{KhavisonPutinarShapiro}.
When $D$ is merely Lipschitz, ${\mathcal K}^*_D$ may no longer be compact and 
$\sigma({\mathcal K}_D^*)$ may contain essential spectrum - a fact that has motivated several analytical and numerical 
studies~\cite{PerfektPutinar_1, HelsingPerfekt, HelsingKangLim, KangLimYu}.
This behavior has been understood quite precisely in the particular case where $D$ is a planar domain with corners:
in their recent work~\cite{PerfektPutinar_2}, K.-M. Perfekt and M. Putinar have characterized
this essential spectrum to be 
\begin{eqnarray*}
\Sess({\mathcal K}^*_D) &=& [\lambda_-, \lambda_+],
\quad \lambda_+ \;=\; -\lambda_- \;=\; \ds\frac{1}{2}(1 - \ds\frac{\a}{\pi}),
\end{eqnarray*}
where $\a$ is the most acute angle of $D$.
In~\cite{BZ}, an alternative proof of this result is given and a connection 
between $\Sess({\mathcal K}^*_D)$ and the elliptic corner singularity 
functions that describe the field $u$ around the corners
is established.\par
\medskip

The main purpose of the present work is to study the spectrum of bowtie-shaped domains 
in 2d~(see Figure \ref{fig.bowtie} below). Metallic bowtie antennas have been the 
subject of extensive experimental studies, as they can produce remarkably large
enhancement of electric fields near their corners, and particularly
in the area of their central neck, which makes them quite interesting in various 
applications, see for instance~\cite{Bi_etal, Cetin, Chen_etal, Ding_etal, Dodson_etal, Lee_etal}.\par
\medskip
In utter rigor, a bowtie-shaped domain $D$ is not Lipschitz regular, since $\partial D$ does not behave as the graph of a
Lipschitz function in the neighborhood of the central point. 
To avoid the tedious issue of introducing a proper definition of the
Poincar\'e-Neumann operator in this context, we take another point of view for characterizing the well-posedness 
of (\ref{eq_cond1}) and thereby the plasmonic resonances of $D$: following the seminal work \cite{KhavisonPutinarShapiro}, we work at the level 
of the so-called Poincar\'e variational
operator $T_D$; see Section \ref{sec.TD}.
For a Lipschitz domain, a simple transformation relates the spectra
of ${\mathcal K}_D^*$ and $T_D$:
\begin{eqnarray*}
\sigma({\mathcal K}^*_D) \;=\; 1/2 - \sigma(T_D)
&\;\textrm{and}\;&
\Sess({\mathcal K}^*_D) \;=\; 1/2 - \Sess(T_D),
\end{eqnarray*}
see for instance~\cite{BZ}. In the context of a bowtie-shaped domain $D$, we prove that the spectrum $\sigma(T_D)$ consists only of essential spectrum, and fills the whole interval $[0,1]$: 
$$ \sigma(T_D) = \Sess(T_D) = [0,1] ;$$
see Theorem \ref{th.specbowtie}.
\par
\medskip
It is also interesting to compare the spectrum of the Poincar\'e variational operator $T_D$ of 
a `true' (non Lipschitz) bowtie-shaped antenna $D$ with that of `quasi' (Lipschitz) bowtie-shaped inclusion $D_\d$ - a version of $D$ 
where the two wings of the bowtie are separated by a small distance $\d > 0$ (see Figure \ref{fig.bowtiectt} below).
The theory about the essential spectrum of the Neumann-Poincar\'e operator of planar domains with corners devised in \cite{PerfektPutinar_2} applies in the latter case, 
with the conclusion that the essential spectrum of the Neumann-Poincar\'e operator $K^*_{D_\d}$ of $D_\d$ is an interval 
$[-\lambda^+, \lambda^+] \Subset [-1/2,1/2]$, where $\lambda^+$
only depends on the value of the angle(s) of each sector
and is independent of~$\d$. 
We show that as $\d \to 0$, $\sigma(K^*_{D_\d})$ cannot reduce to its
essential spectrum and must contain eigenvalues in the range
$]-1/2,\lambda^-[ \cup ]\lambda^+, 1/2[$.
These eigenvalues become denser and denser in that set as $\d \to 0$. 
This phenomenon was already observed in~\cite{HelsingMcPhedranMilton}
(see also~\cite{Milton} pp. 378-379)
for the related problem of finding the spectrum of the effective permittivity 
of a composite made of square inclusions of a metamaterial
embedded in a dielectric background medium.
See in particular the computations reported in~\cite{HelsingMcPhedranMilton},
and the associated movies~\cite{Helsing_computations}
which show how eigenvalues become denser as the distance between 
the corners of the square inclusions tends to $0$.
The spectrum considered in~\cite{HelsingMcPhedranMilton} is 
closely related to ours: 
see~\cite{BDT} that studies the homogenization limit of the spectrum of 
the Neumann-Poincar\'e operator.\par
\medskip

The present article is organized as follows. The setting and notations
are described in Section \ref{sec_2}, where some background material about plasmonic resonances and the Poincar\'e variational operator is briefly recalled. 
In Section~\ref{sec_3}, we construct
corner singularity functions that describe the behavior of solution
to the transmission problem (\ref{eq_cond1}) near the central neck of a bowtie-shaped domain $D$ when the
permittivity inside $D$ is negative.
Contrarily to the case of connected planar domains with corners 
(see~\cite{Bonnet_etal_1, Bonnet_etal_3, BZ})
these functions always lie outside the energy space $H^1$.
In Section \ref{sec_4}, we use these singular functions to prove that the spectrum of 
$T_D$ is composed only of essential spectrum and occupies the whole interval $[0,1]$.
In Sections~\ref{sec_5} and~\ref{sec_6}, we relate this behavior to
that of the spectrum of a near-bowtie shaped domain $D_\d$, as $\d \to 0$.
This article ends with the short Appendix \ref{sec.weyl} recalling some material about Weyl sequences.

\section{The Poincar\'e variational operator of a bowtie-shaped plasmonic antenna} \label{sec_2}

\subsection{Generalities about plasmonic resonances}\label{sec.plasm}

Let $\Omega \subset \R^2$ denote a bounded open set with smooth boundary,
containing the origin. 
Throughout the article, a point $x \in \R^2$ shall be indifferently represented in terms of its Cartesian coordinates $x = (x_1,x_2)$ 
or its polar coordinates with origin $0$, as $x = (r,\theta)$. 
Also, for $\rho >0$, we denote by $B_\rho$ (resp. $B_\rho(x)$) 
the open ball with center $0$ (resp. $x$) and radius $\rho$. \par
 
\medskip 

Let $D \Subset \Omega$ be an open set, representing an inclusion in $\Omega$; for the moment, no particular assumption is made about the regularity of $D$. 
As we have hinted at in the introduction, the plasmonic resonances of the inclusion $D$ are described in terms of the conduction equation for 
the voltage potential $u$:
\begin{equation} \label{eq_cond}
\left\{ \begin{array}{ccll}
-\textrm{div}(a(x) \nabla u(x)) &=& f &\textrm{in}\; \Omega,
\\
u(x) &=& 0 & \textrm{on}\; \partial \Omega,
\end{array} \right.
\end{equation}
where $f$ is a source in $H^{-1}(\Omega)$, and the conductivity $a(x)$ is piecewise constant:
\begin{eqnarray} \label{def_a}
a(x) &=&
\left\{ \begin{array}{cl}
k \in {\mathbf C}& \quad x \in D,
\\
1 & \quad x \in \Omega \setminus \overline{D}.
\end{array} \right.
\end{eqnarray}

Classical results from the theory of elliptic PDE's show that when 
$k \in \mathbf{C} \setminus \R^-$, the equation (\ref{eq_cond}) 
has a unique solution $u\in H^1_0(\Omega)$, which satisfies: 
\begin{eqnarray*}
||u||_{H^1_0(\Omega)} &\leq& C(k) \, ||f||_{H^{-1}(\Omega)},
\end{eqnarray*}
where the constant $C(k) > 0$ depends on $k$.
In the above relation, and throughout this article, the space $H^1_0(\Omega)$ is equipped with the following inner product and
associated norm
\begin{eqnarray*}
<u,v>_{H^1_0(\Omega)} \;=\; \ds\int_\Omega \nabla u \cdot \nabla v \,dx,
&\text{and}&
||u||_{H^1_0(\Omega)} \;=\; 
\left(\ds\int_\Omega |\nabla u|^2 \,dx
\right)^{1/2}.
\end{eqnarray*}
\medskip

Our main purpose is to describe the quasistatic \textit{plasmonic resonances}
of $D$; these are defined as the values $k \in \mathbf{C}$ of the conductivity
inside $D$ such that there exists a bounded sequence $f_n$ of sources in $H^{-1}(\Omega)$ - say $\lvert\lvert f_n\lvert\lvert_{H^{-1}(\Omega)} = 1$ - 
such that there exists a sequence $u_n$ of associated voltage potentials, solution to (\ref{eq_cond}), which blows up: $\lvert\lvert u _n\lvert\lvert_{H^1_0(\Omega)} \to \infty$ as $n\to \infty$.
\par\medskip

\begin{remark}
In our setting, the considered inclusion $D$ is embedded in a large (yet bounded) `hold-all' 
domain $\Omega$, and not in the free space $\R^2$ as is customary in the study of the Neumann-Poincar\'e operator (see e.g. \cite{BonnetierTriki_2,KhavisonPutinarShapiro}). 
This is only a matter of simplicity, since we intend to focus on the properties of $D$ and not on those of its surrounding environment. The present 
study could easily be adapted to the case where $\Omega = \R^2$, 
by using weighted Sobolev spaces instead of $H^1_0(\Omega)$ as energy space.
\end{remark}

\subsection{The Poincar\'e variational operator, and its connection with the Neumann-Poincar\'e operator in the case of a Lipschitz inclusion}

\subsubsection{The Poincar\'e variational operator and the conduction equation}\label{sec.TD}

Following the lead of the seminal work \cite{KhavisonPutinarShapiro}, 
a convenient tool in our study of the plasmonic resonances of $D$ is the Poincar\'e variational operator $T_D : H^1_0(\Omega) \to H^1_0(\Omega)$, 
defined as follows: for $u \in H^1_0(\Omega)$, $T_Du $ is the unique function in $H^1_0(\Omega)$ such that:
\begin{eqnarray} \label{def_T}
\forall \; v \in H^1_0(\Omega), \quad
\ds\int_\Omega \nabla (T_D u) \cdot \nabla v \,dx
&=& 
\ds\int_D \nabla u \cdot \nabla v \,dx.
\end{eqnarray}
The link between $T_D$ and the conduction equation (\ref{eq_cond}) is the following: a simple calculations shows 
that $u \in H^1_0(\Omega)$ satisfies (\ref{eq_cond}) if and only if:
\begin{equation}\label{eq.eqconducTD}
 (\beta \Id - T_D) u = \beta g, \text{ where }\beta := \frac{1}{1-k}
 \end{equation}
and where $g$ is obtained from $f$ via the Riesz representation theorem
$$ 
\forall v \in H^1_0(\Omega), \:\: \int_\Omega{\nabla g \cdot \nabla v \:dx} = \langle f,v \rangle_{H^{-1}(\Omega),H^1_0(\Omega)}.
$$

In the same spirit, the Poincar\'e variational operator offers a convenient characterization of the plasmonic resonances of $D$:
\begin{proposition} \label{prop_plasm_res}
Let $k\in \mathbf{C}$, $k\neq 1$, and let the conductivity $a(x) \in L^\infty(\Omega)$ 
be defined as (\ref{def_a}). 
The following statements are equivalent:
\begin{itemize}
\item[1.] There exists a sequence $u_n \in H^1_0(\Omega)$ 
such that
\begin{eqnarray} \label{cond_1}
|| \: \div(a \nabla u_n) \:||_{H^{-1}(\Omega)} \;=\; 1
&\textrm{and}\;& ||u_n||_{H^1_0(\Omega)} \to \infty.
\end{eqnarray}

\item[2.] The conductivity $k$ inside $D$ is such that
$\beta := \ds\frac{1}{1-k}$ belongs to the spectrum $\sigma(T_D)$ of $T_D$.
\end{itemize}
\end{proposition}
\begin{proof}
Let us first assume that $\beta = \frac{1}{k-1}$ is in $\sigma(T_D)$. 
By the Weyl criterion - see Theorem \ref{th.weyl} in Appendix \ref{sec.weyl} - there exists a sequence $u_n \in H^1_0(\Omega)$ such that: 
$$ \lvert\lvert u_n \lvert\lvert_{H^1_0(\Omega)} = 1 \text{ and } \lvert\lvert T_D u_n - \beta u_n \lvert\lvert_{H^1_0(\Omega)}  \xrightarrow{n \to \infty} 0.$$
Up to making a small perturbation of the $u_n$, one may additionnally assume that $ \lvert\lvert T_D u_n - \beta u_n \lvert\lvert_{H^1_0(\Omega)} \neq 0$ for all $n$. 
Now, let $v_n := \frac{1}{ (k-1)\lvert\lvert T_D u_n - \beta u_n \lvert\lvert_{H^1_0(\Omega)}} u_n$; obviously, $\lvert\lvert v_n \lvert\lvert_{H^1_0(\Omega)} \to \infty$ as $n \to \infty$, while the definition of $T_D$ implies:
$$\begin{array}{>{\displaystyle}cc>{\displaystyle}l}
\lvert\lvert \div(a\nabla v_n) \lvert\lvert_{H^{-1}(\Omega)} &=&  \sup\limits_{w \in H^1_0(\Omega), \atop \lvert \lvert w \lvert\lvert_{H^1_0(\Omega)} = 1}{\int_{\Omega}{a(x)\nabla v_n \cdot \nabla w \:dx}} \\
&=& (k-1) \sup\limits_{w \in H^1_0(\Omega), \atop \lvert \lvert w \lvert\lvert_{H^1_0(\Omega)} = 1}{\int_{\Omega}{\nabla(T_D v_n - \beta v_n) \cdot \nabla w \:dx}}, \\
&=&1.
\end{array}$$
Hence, the sequence $v_n$ satisfies (\ref{cond_1}). \par\medskip

Conversely, if there exists a sequence $u_n \in H^1_0(\Omega)$ such that (\ref{cond_1}) holds, a similar argument allows to 
construct a Weyl sequence for $T_D$ and the value $\beta = \frac{1}{1-k}$, so that $\beta$ belongs to $\sigma(T_D)$. This concludes the proof. 
\end{proof}

We may therefore look for the plasmonic resonances of the inclusion $D \Subset \Omega$ by searching for the values of the conductivity
$k \in \mathbf{C}$ inside $D$ such that $\beta = \frac{1}{1-k} \in \sigma(T_D)$.
This remark motivates the study of the spectrum $\sigma(T_D)$. 

\subsubsection{Structure of the spectrum of the Poincar\'e variational operator of a Lipschitz regular inclusion}\label{sec.sPV0lip}

In this section, we assume $D$ to be Lipschitz regular; for further purpose, we also allow $D$ to consist of several connected components: $D = \bigcup_{i=1}^N{D_i}$, $i=1,...,N$. 
The following proposition outlines the general structure of the spectrum $\sigma(T_D)$;
see~\cite{BDT} for a proof.

\begin{proposition} \label{prop_TDd}
The operator $T_{D}$ is bounded, self-adjoint, with operator norm $||T_{D}||=1$.
Moreover, 
\begin{itemize}
\item[(i)]
Its spectrum $\sigma(T_{D})$ is contained in the interval $[0,1]$.

\item[(ii)] The eigenspace associated to the eigenvalue $0$ is:
\begin{eqnarray*}
\Ker(T_{D}) &=&
\{ u \in H^1_0(\Omega),\: \exists c_i \in \R, \: u = c_i \text{ in } D_i, i=1,...,N \}.
\end{eqnarray*}

\item[(iii)] The value $1$ belongs to $\sigma(T_{D})$ and the associated eigenspace is:
\begin{eqnarray*}
\Ker(\Id-T_{D}) &=&
\{ u \in H^1_0(\Omega), u = 0 \;\textrm{ in }  \Omega \setminus \overline{D_\d} \}; 
\end{eqnarray*}
and $\Ker(\Id - T_{D})$ can be identified with $H^1_0(D)$.

\item[(iv)]
The space $H^1_0(\Omega)$ has the orthogonal decomposition:
\begin{eqnarray}\label{eq.decHd}
H^1_0(\Omega) &=&
\Ker(T_{D}) \oplus \Ker(\Id-T_{D}) \oplus {\mathcal H}_D,
\end{eqnarray}
where ${\mathcal H}_D$, the `non trivial' part of $\sigma(T_D)$, is the closed subspace of $H^1_0(\Omega)$ defined by
\begin{eqnarray} \label{eq.defcalH}
{\mathcal H}_D &=&
\left\{ u \in H^1_0(\Omega), \: \Delta u = 0 \;
\textrm{in}\; D \cup (\Omega \setminus \overline{D}) \text{ and }
\ds\int_{\partial D_{i}} \ds\frac{\partial u^+}{\partial \nu} \, ds = 0, \:\: i=1,...,N \right\}.
\end{eqnarray}
\end{itemize}
\end{proposition}
In the above proposition, we have denoted by $\nu$ the unit normal vector to the Lipschitz boundary $\partial D$ pointing outward $D$;
for a.e. $x \in \partial D$ and for any smooth enough function $w$, the traces $w^\pm$ and normal derivatives $\frac{\partial w^\pm}{\partial \nu}$ of $w$ are respectively defined by:
\begin{eqnarray*}
w^\pm(x) &=& \lim_{t \to 0 \atop t > 0} w(x \pm t \nu(x)), \text{ and } \frac{\partial w^\pm}{\partial \nu}(x) = \lim_{t\to 0 \atop t >0}{\nabla w(x\pm t\nu(x)) \cdot \nu(x)}. 
\end{eqnarray*}
Note that these identities have to be considered in the weaker sense of traces - in $H^{1/2}(\partial D)$ and $H^{-1/2}(\partial D)$ respectively - if less regularity is assumed on $w$, as is the case in (\ref{eq.defcalH}). 

\subsubsection{Connection with the Neumann-Poincar\'e operator when $D$ is Lipschitz}\label{sec.NPOlip}

In this section again, we assume $D$ to be Lipschitz regular. 
As we have mentionned in the introduction, the operator $T_D$ has then close connections with the Neumann-Poincar\'e operator ${\mathcal K}_D^* : H^{-1/2}(\partial D) \to H^{-1/2}(\partial D)$ of 
the inclusion $D$, which we now briefly recall. \par
\medskip 

Let $P(x,y)$ denote the Poisson kernel associated to $\Omega$, defined by
\begin{eqnarray*}
P(x,y) &=& G(x,y) + R_x(y), \quad x,y \in \Omega, x \neq y,
\end{eqnarray*}
where $G(x,y)$ is the Green function in the two-dimensional free space:
\begin{eqnarray*}
G(x,y) &=& \ds\frac{1}{2\pi} \log|x-y|,
\end{eqnarray*}
and for a given $x \in \Omega$, $R_x(y)$ is the smooth solution to
\[
\left\{ \begin{array}{clcl} 
\Delta_y R_x(y) &=& 0 & y \in \Omega,
\\
R_x(y) &=& -G(x,y) & y \in \partial \Omega, 
\end{array} \right.
\]
see for instance \cite{AmmariKang}. 
Thence ,the single layer potential
$S_D\vf $ of a function $\vf \in L^2(\partial D)$ is defined by
\begin{eqnarray*}
S_D\vf(x) &=&
\int_{\partial D} P(x,y) \vf(y)\,ds(y),
\quad x \in D \cup (\Omega \setminus \overline{D}).
\end{eqnarray*}
It is well-known~\cite{Folland,Verchota} that $S_D\vf$ belongs to the space $\mathfrak{h}_D$ defined by
$$ Ê\mathfrak{h}_D := \left\{Êu \in H^1_0(\Omega), \:\: \Delta u = 0 \text{ in } D \cup (\Omega \setminus \overline{D}) \right\};$$
notice that $\mathfrak{h}_D$ is slightly larger than its subspace ${\mathcal H}_D$ defined in (\ref{eq.defcalH}) (they differ by a finite-dimensional space).
Additionally, the definition of $S_D$ extends to potentials $\varphi \in H^{-1/2}(\partial D)$ \cite{mclean}, 
and the induced mapping $S_D : H^{-1/2}(\partial D) \to \mathfrak{h}_D$ is an isomorphism \cite{BonnetierTriki_2}.

The normal derivatives of the single layer potential across $\partial D$ satisfy the Plemelj jump conditions
\begin{eqnarray} \label{jump_S}
\ds\frac{ \partial S_D \vf}{\partial \nu}^\pm
&=&
\left(\pm\ds\frac{1}{2}\Id + {\mathcal K}^*_D \right)\vf,
\end{eqnarray}
where ${\mathcal K}^*_D : L^2(\partial D) \to L^2(\partial D)$ is the Neumann-Poincar\'e operator of $D$, 
defined by
\begin{eqnarray*}
{\mathcal K}^*_D\vf(x) &=&
\ds\int_{\partial D} \ds\frac{\partial P}{\partial \nu_x}(x,y)
\vf(y) \,ds(y),
\end{eqnarray*}
whose definition makes sense for Lipschitz domains~\cite{Coifman_etal,Verchota}. 
In turn, ${\mathcal K}_D^*$ extends as an operator
$H^{-1/2}(\partial D) \rightarrow H^{-1/2}(\partial D)$; see \cite{mclean}.

\medskip 

Eventually, the Poincar\'e variational operator $T_D: \mathfrak{h}_D \to \mathfrak{h}_D$ and the Neumann-Poincar\'e operator ${\mathcal K}_D^* : H^{-1/2}(\partial D) \to H^{-1/2}(\partial D)$ are related as: 
$$ R_D = - S_D \circ {\mathcal K}_D^* \circ S_D^{-1}, \text{ where }ÊR_D := T_D - \frac{1}{2}\Id;$$
see \cite{BonnetierTriki_2,KhavisonPutinarShapiro}. In particular, the spectra of $T_D$ and ${\mathcal K}_D^*$ are equal, up to a constant shift: 
$$
\sigma({\mathcal K}^*_D) \;=\; 1/2 - \sigma(T_D)
\textrm{ and }
\Sess({\mathcal K}^*_D) \;=\; 1/2 - \Sess(T_D),$$
and the plasmonic resonances of $D$ may be equivalently studied from the vantage of $T_D$ or ${\mathcal K}_D^*$.

\subsection{The case of a bowtie-shaped antenna}\label{sec.prespb}

\subsubsection{Presentation of the physical setting}\label{sec.physbt}

From now on and in the remaining of this article, we assume that $D$ 
is shaped as a bowtie (and hence is not Lipschitz)~:
$D = D_1 \cup D_2$ is the reunion of two connected domains
whose boundaries are smooth except at $0$, and there exist $r_0 > 0$ and $0 < \a < \pi$ such that:
\begin{eqnarray*}
D_1 \cap B_{r_0} &=& 
\{ (r\cos \theta, r\sin\theta), \:0 < r < r_0, \: -\a/2 \leq \theta \leq \a/2 \},
\\
D_2 \cap B_{r_0} &=& 
\{(r\cos\theta, r\sin\theta), \: 0 < r < r_0,  \:\pi-\a/2 \leq \theta \leq \pi+\a/2 \};
\end{eqnarray*}
see Figure \ref{fig.bowtie}.
We refer to $D_1$ and $D_2$ as the `wings' of the bowtie
(after all, 'bowtie' translates as `n\oe ud papillon' in French).

\begin{figure}[!ht]
\centering
\includegraphics[width=0.6\textwidth]{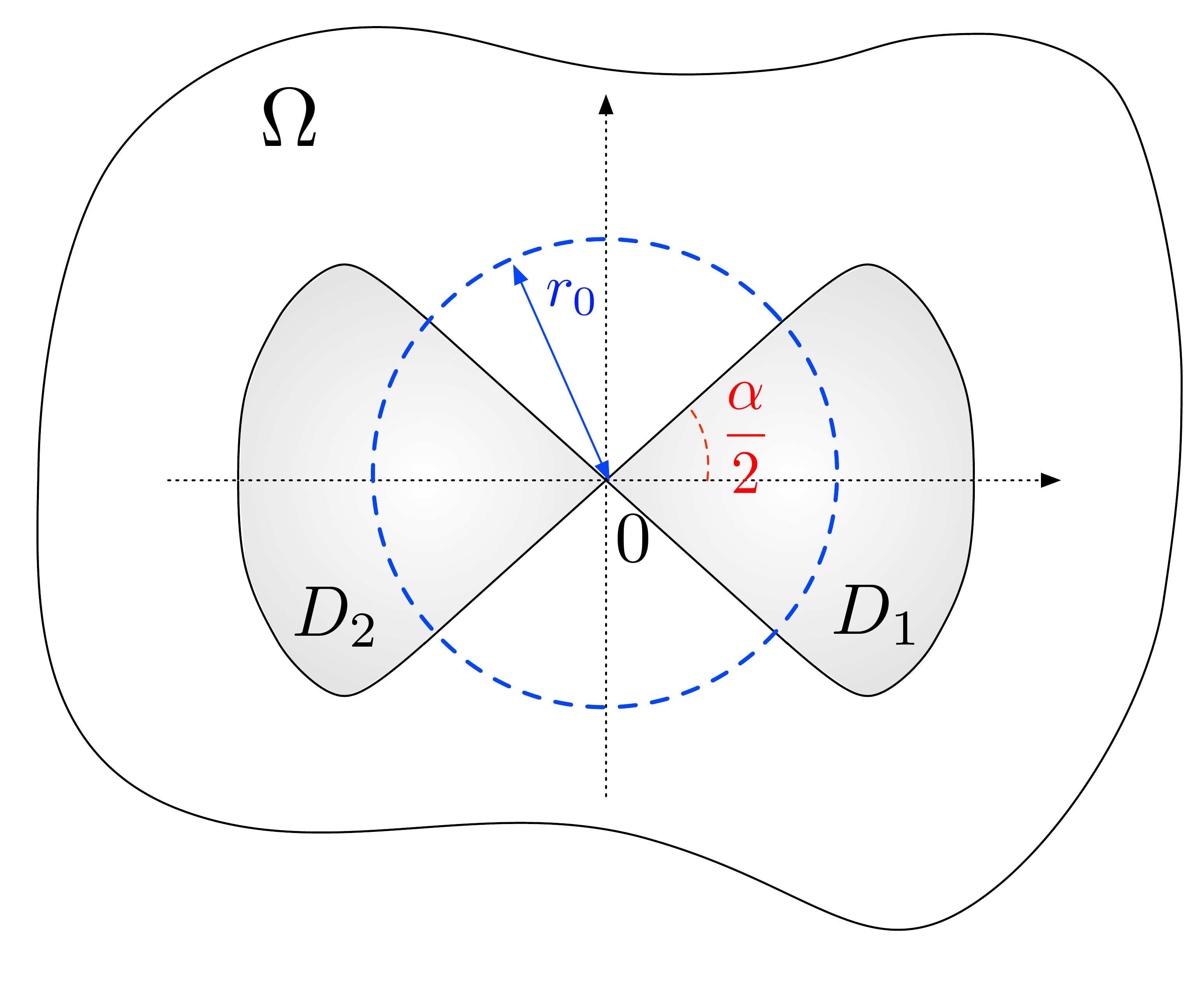}
\caption{\it Setting of the bowtie-shaped domain presented in Section \ref{sec.prespb}.}
\label{fig.bowtie}
\end{figure}

\begin{remark}
We have assumed $D$ to be smooth except at the contact point $0$ between the
wings $D_1$ and $D_2$. Our analysis remains valid if 
$D_1$ and $D_2$ have additionnal corners (for instance if they are shaped as triangles,
as is often the case in actual physical devices). Indeed, as we show below, it is 
the contact point between the two wings that carries the worst singularity and determines the width of the essential spectrum of the Poincar\'e variational 
operator of $D$.
\end{remark}

\subsubsection{The Poincar\'e variational operator of a bowtie-shaped antenna}\label{sec.TD}

The bowtie-shaped domain $D$ of Section \ref{sec.prespb} 
fails to be Lipschitz regular, since it does not arise as 
the subgraph of a Lipschitz function in the vicinity of the point $0$.
Rather than defining and studying an adapted Neumann-Poincar\'e operator (see \cite{AmBonTriVog} however for a related construction),
we base our study of the well-posedness of~(\ref{eq_cond}) on the Poincar\'e
variational operator, whose definition (\ref{def_T}) naturally makes sense in the case of domains like $D$.\par
\medskip

Since the set $D$ is not Lipschitz regular, some care is in order about the definition of the attached functional spaces.
We denote by $H^1(D)$ is the set of functions on $D$ which are restrictions 
to $D$ of functions in $H^1(\R^2)$ and by $H^1_0(D)$ the closure of 
${\mathcal C}^\infty_c(D)$ in $H^1_0(\Omega)$.
Also, $\widetilde{H}^1(D)$ is the set of functions $u\in H^1(D)$ whose extension to $\Omega$ by $0$ is in $H^1_0(\Omega)$. 
Let us recall that, if ${\mathcal O} \Subset \Omega$ is a Lipschitz domain 
$\widetilde{H}^1({\mathcal O} ) = H^1_0({\mathcal O} )$; see \cite{Grisvard}. Unfortunately, the bowtie-shaped domain $D$ is not Lipschitz, but this property 
nevertheless holds, as we now prove:

\begin{lemma}\label{lem.ht}
Let $D$ be a bowtie as described in Section \ref{sec.physbt}.
Then $\widetilde{H}^1(D) = H^1_0(D)$.
\end{lemma}
\begin{proof}
On the one hand, any smooth function in ${\mathcal C}^\infty_c(D)$ can be extended by $0$
to a function in $H^1_0(\Omega)$, so that by density, $H^1_0(D) \subset \widetilde{H}^1(D)$ (this inclusion actually holds true in the case of a general domain $D$).
\par
\medskip
On the other hand, to show the reverse inclusion, let
$u \in \widetilde{H}^1(D)$; given the particular shape of $D$, one may write $u = u_1 + u_2$, for some $u_1,u_2 \in H^1(D)$
with $\textrm{Supp}(u_1) \subset D_1$ and $\textrm{Supp}(u_2) \subset D_2$.
Since $D_1$ is a Lipschitz domain, $u_1 \in \widetilde{H}^1(D_1) = H^1_0(D_1)$ and
$u_1$ arises as the limit in $H^1_0(\Omega)$ of a sequence of functions $u_{1,n} \in {\mathcal C}^\infty_c(D_1)$; hence $u_1 \in H^1_0(D)$. 
Similarly, $u_2 \in H^1_0(D)$, so that $\widetilde{H}^1(D) \subset H^1_0(D)$. 
\end{proof}

The main spectral properties of $T_D$ are described in the following proposition, which is an echo of Proposition \ref{prop_TDd} in the case of the bowtie-shaped domain $D$.
The proof is essentially that of Proposition 3.2 in \cite{BDT}, except 
for a technical point that we make precise.

\begin{proposition} \label{prop_TD}
The operator $T_D$ is bounded, self-adjoint, with operator norm $||T_D||=1$.
Moreover, 
\begin{itemize}
\item[(i)]
Its spectrum $\sigma(T_D)$ is contained in the interval $[0,1]$.

\item[(ii)] The eigenspace associated to the eigenvalue $0$ is:
\begin{eqnarray*}
\Ker(T_D) &=&
\{ u \in H^1_0(\Omega),\: \exists c \in \R, \: u = c \text{ in } D \}.
\end{eqnarray*}

\item[(iii)] The value $1$ belongs to $\sigma(T_D)$ and the associated eigenspace is
\begin{eqnarray*}
\Ker(\Id-T_D) &=&
\{ u \in H^1_0(\Omega), u = 0 \;\textrm{ in }  \Omega \setminus \overline{D} \}; 
\end{eqnarray*}
therefore, in light of Lemma \ref{lem.ht},
$\Ker(\Id - T_D)$ is naturally identified with $H^1_0(D)$.

\item[(iv)]
The space $H^1_0(\Omega)$ decomposes as
\begin{eqnarray*}H^1_0(\Omega) &=&
\Ker(T_D) \oplus \Ker(\Id-T_D) \oplus {\mathcal H}_D,
\end{eqnarray*}
where ${\mathcal H}_D$ is the closed subspace of $H^1_0(\Omega)$ defined by
\begin{eqnarray} \label{eq.defcalHD}
{\mathcal H}_D &=&
\left\{ u \in H^1_0(\Omega), \: \Delta u = 0 \;
\textrm{in}\; D \cup (\Omega \setminus \overline{D}) \text{ and }
\ds\int_{\partial D_1 \cup \partial D_2} \ds\frac{\partial u^+}{\partial \nu} \, ds = 0 \right\}.
\end{eqnarray}
\end{itemize}
\end{proposition}

\begin{proof}
\textit{(i):} It is a straightforward consequence of the self-adjointness of $T_D$ and
of the fact that $\lvert\lvert T_D \lvert\lvert =1$.
\\

\noindent\textit{(ii):} 
By definition, a function $u \in H^1_0(\Omega)$ belongs to $\Ker(T_D)$ if and only if
$$ \forall v \in H^1_0(\Omega), \:\: \int_D{\nabla u \cdot \nabla v \:dx} = 0.$$
Let $u \in \Ker(T_D)$; then $\int_D{\lvert \nabla u \lvert^2 \:dx} = 0$, 
so that $u$ is constant on $D_1$ and on $D_2$: 
there exist $c_1$, $c_2 \in \R$ such that $u = c_i$ on $D_i$, $i=1,2$. Moreover, since $u \in H^1_0(\Omega)$, the trace $u\lvert_\ell$ 
of $u$ on the one-dimensional subset $\ell := \left\{Êx = (x_1,x_2) \in \Omega, \: x_2 = 0 \right\}$ belongs to $H^{\frac{1}{2}}(\ell)$. 
However, by the definition of $u$ and $D$, there exists $r_0 >0$ such that: 
$$ u\lvert_\ell(x) = c_1 \text{ if }Êx =(x_1,0) \text{ with } -r_0 < x_1 < 0 \text{ and } u\lvert_\ell(x) = c_2 \text{ if }Êx =(x_1,0) \text{ with } 0 < x_1 < r_0.$$
This implies that $c_1 = c_2$. 
Conversely, if $u\in H^1_0(\Omega)$ satisfies $u=c$ on $D$ for some $c \in \R$, then $u \in \Ker(T_D)$. \\

\noindent \textit{(iii):} This follows from a similar argument. \\

\noindent \textit{(iv):} A function $u \in H^1_0(\Omega)$ is orthogonal to $\Ker(T_D)$ if and only if
\begin{equation}\label{eq.orthoker}
\forall v \in \Ker(T_D), \:\: \int_\Omega{\nabla u \cdot \nabla v \:dx} = 0.
\end{equation}
Using first test functions $v \in {\mathcal C}^\infty_c(\Omega \setminus \overline{D})$, we obtain that $\Delta u = 0$ in $\Omega \setminus \overline{D}$. 
Now using arbitrary functions $v \in H^1_0(\Omega)$ take a constant value inside $D$, and integrating by parts yields the subsequent condition:
\begin{equation}\label{eq.subscond}
\int_{\partial D_1 \cup \partial D_2}{\frac{\partial u^+}{\partial n} \:ds} = 0.
\end{equation}
Eventually, one proves in a similar way that $u \in H^1_0(\Omega)$ is orthogonal to $\Ker(\Id-T_D)$ if and only if
$ \Delta u = 0 \text{ in } D$.
\end{proof}

\begin{remark}\noindent
\begin{itemize}
\item Rigorously speaking, the definition of the normal derivative $\frac{\partial u^+}{\partial \nu}$
as an element in $H^{-1/2}(\partial D)$ in (\ref{eq.defcalHD}) is not so straightforward in the present context, 
since $D$ fails to be Lipschitz. It is possible to define this notion nevertheless, but
we shall not require this in the present article; for our purpose, we may 
understand (\ref{eq.subscond}) in the sense that (\ref{eq.orthoker}) holds for any function $v \in H^1_0(\Omega)$ such that $v \equiv 1$ on $D$.
\item Interestingly, from the vantage of the eigenspaces of $T_D$, $D$ behaves as if it were a connected domain (compare Proposition \ref{prop_TD} with its the counterpart Proposition \ref{prop_TDd} in the Lipschitz case). This peculiarity highlights one specificity of bowtie-shaped domains.
\end{itemize}
\end{remark}



\section{Corner singularity functions for a bowtie} \label{sec_3}

In this section we characterize the local behavior of solutions to the equation
\begin{equation} \label{eq_cond0}
\textrm{div}(a\nabla u) = 0, \text{ where }Êa(x) \text{ is given by (\ref{def_a}),}
\end{equation} 
in the vicinity of the contact point $x=0$ of the two wings of the bowtie $D$.\par
\medskip
When $k$ takes a positive real value, this question pertains to the theory of elliptic corner singularity, to which a great deal of literature is devoted,
see e.g. \cite{Kondratiev,Grisvard,CostabelDaugeNicaise,Dauge,KozlovMazyaRossmann}.
In a nutshell, for a two-phase transmission problem of the form (\ref{eq_cond0}) featuring a piecewise smooth inclusion with corners, 
$u$ is expected to decompose as the sum of a regular 
and of a singular part $u = u_{\text{\rm reg}} + u_{\text{\rm sing}}$, where $u_{\text{\rm reg}}$ has at least $H^2$ regularity,
whereas $u_{\text{\rm sing}}$ is $H^1$ but not $H^2$ regular.
Moreover, in the neighborhood of a corner, the latter function takes the following form in polar coordinates:
\begin{eqnarray*}
u_{\text{\rm sing}}(r,\theta) &=& C r^\eta \vf(\theta).
\end{eqnarray*}
In this expression, $C$ is a multiplicative constant, $\eta \in (0,1]$ and $\vf$ is a piecewise smooth function;
both $\eta$ and $\vf$ depend on the geometry of the wedge and of the contrast in conductivities.
\medskip

In the present section, we investigate the local behavior of the non trivial solutions to (\ref{eq_cond0}) in the case of a bowtie-shaped domain $D$, when $k$ takes negative values.
More precisely, let the conductivity $a$ be defined by:
\begin{eqnarray*}
a(\theta) &:=& \left\{ \begin{array}{ll}
k & \textrm{if}\; |\theta| < \frac{\a}{2} \;\textrm{or}\; |\pi - \theta| < \frac{\a}{2}
\\
1 & \textrm{otherwise}.
\end{array} \right.
\end{eqnarray*}
We search for a solution to~(\ref{eq_cond0}) in the whole space $\R^2$.
More specifically, we are interested in finding {\em some} solutions to~(\ref{eq_cond0}) in the sense of distributions
which do not belong to the energy space $H^1_{\text{\rm loc}}(\R^2)$. These solutions will be the key ingredient in the construction of generalized
eigenfunctions of $T_D$ carried out in Section~\ref{sec_4}.
Considering the symmetry of the geometric configuration 
with respect to the horizontal axis, it is enough to search for solutions $u$ 
to one of the following two problems set on the upper half-space 
$\Pi^+ := \left\{ x = (x_1,x_2) \in \R^2, \:\: x_2 >0 \right\}$: 
\begin{eqnarray}
\left\{ \begin{array}{ccll}
\textrm{div}(a \nabla u) &=& 0 &\textrm{in}\; \Pi^+
\\
u(x) &=& 0 & \text{on }Ê\partial \Pi^+,
\end{array} \right.
\label{eq_D}
\\
\left\{ \begin{array}{ccll}
\textrm{div}(a \nabla u) &=& 0 &\textrm{in}\; \Pi^+
\\
\frac{\partial u}{\partial n}(x) &=& 0 & \text{on }Ê\partial \Pi^+.
\end{array} \right.
\label{eq_N}
\end{eqnarray}
Indeed, assume that $u_D$ is a solution to (\ref{eq_D}) in the sense of distributions, and 
define 
$$ u(x_1,x_2) = \left\{ 
\begin{array}{cl}
u_D(x_1,x_2) & \text{if } x_2 \geq 0, \\
-u_D(x_1,-x_2) & \text{if } x_2 < 0,
\end{array}
\:\: \text{a.e. } x = (x_1,x_2) \in \R^2.
\right.$$
Then it is easily seen that $u$ is a solution to (\ref{eq_cond0}) in the sense of distributions. Likewise, 
if $u_N$ is a solution to (\ref{eq_N}), then 
$$ u(x_1,x_2) := \left\{ 
\begin{array}{cl}
u_N(x_1,x_2) & \text{if } x_2 \geq 0, \\
u_N(x_1,-x_2) & \text{if } x_2 < 0.Ê
\end{array}
\right.
$$ 
solves~(\ref{eq_cond0}). 

\medskip
Let us first search for a solution to (\ref{eq_D}) under the form 
$u(r,\theta) = r^{i\xi} \vf(\theta)$ for some $\xi >0$ and some function $\varphi(\theta)$ which is $2\pi$-periodic.  
Simple calculations show that (\ref{eq_D}) implies:
$$
(a(\theta) \vf^\prime(\theta))^\prime - \xi^2 a(\theta) \vf(\theta) \;=\; 0
$$
and so $\vf$ has the form:
\begin{eqnarray*}
\vf(\theta) &=& 
\left\{ \begin{array}{ll}
a_1 \cosh(\xi \theta) + b_1 \sinh(\xi \theta) & \quad 0 < \theta <\frac{\a}{2},
\\
a_2 \cosh(\xi \theta) + b_2 \sinh(\xi \theta) & \quad \frac{\a}{2} < \theta < \pi -\frac{\a}{2},
\\
a_3 \cosh(\xi \theta) + b_3 \sinh(\xi \theta) & \quad \pi - \frac{\a}{2} < \theta < \pi,
\end{array} \right.
\end{eqnarray*}
for some constants $a_j$, $b_j$, $j=1,2,3$ to be determined. 
Now expressing the transmission and boundary conditions in (\ref{eq_D}) yields a homogeneous
linear system for the coefficients $a_j,$ $b_j$. 
Existence of a non-trivial solution to (\ref{eq_D}) requires that the determinant of 
this system should vanish. A straightforward calculation shows that the latter determinant is the following polynomial of order $2$ in $k$:
\begin{equation} \label{det_D}
d_D(k) = \cosh^2(\xi \a/2) \sinh[\xi(\pi - \a)] k^2
+ \cosh^2[\xi (\pi - \a)] \sinh(\xi\a) k  + \sinh^2(\xi \a/2) \sinh[\xi(\pi - \a)],
\end{equation}
in which $\xi$ acts as a parameter. The roots of $d_D(k)$ are: 
\begin{equation}
k_{D,+}(\xi) = \ds\frac{ -\left( \cosh[\xi(\pi-\a)] - 1 \right) \sinh(\xi \a)}
{2 \cosh^2(\xi \a/2) \sinh[ \xi(\pi -\a)] } \text{ and }
k_{D,-}(\xi) = \ds\frac{ -\left( \cosh[\xi(\pi-\a)] + 1 \right) \sinh(\xi \a)}
{2 \cosh^2(\xi \a/2) \sinh[ \xi(\pi -\a)] }.
\end{equation}

Likewise, there exists a solution to (\ref{eq_N}) of the form $u = r^{i\xi} \varphi(\theta)$ provided the following determinant vanishes:
\begin{equation} \label{det_N}
d_N(k) = \sinh^2(\xi \a/2) \sinh[\xi(\pi - \a)] k^2
+ \cosh^2[\xi (\pi - \a)] \sinh(\xi\a) k  + \cosh^2(\xi \a/2) \sinh[\xi(\pi - \a)].
\end{equation}
Its roots are:
\begin{equation*}
k_{N,+}(\xi) = \ds\frac{ -\left( \cosh[\xi(\pi-\a)] - 1 \right) \sinh(\xi \a)}
{2 \sinh^2(\xi \a/2) \sinh[ \xi(\pi -\a)] } \text{ and }
k_{N,-}(\xi) = \ds\frac{ -\left( \cosh[\xi(\pi-\a)] + 1 \right) \sinh(\xi \a)}
{2 \sinh^2(\xi \a/2) \sinh[ \xi(\pi -\a)] }.
\end{equation*}
It is easy to check that $k_{D,+}$ is a smooth function on $\R_+$, that $\lim_{\xi \to 0^+} k_{D,+}(\xi) = 0$, while
$\lim_{\xi \to +\infty} k_{D,+}(\xi) = -1$. 
In addition, we may rewrite:
\begin{eqnarray*}
k_{D,+}(\xi) &=&
\left( -\ds\frac{\cosh[\xi(\pi - \a)] -1}{\sinh[\xi(\pi - a)]} \right)
\tanh(\xi \a/2),
\end{eqnarray*}
and check that as functions of $\xi > 0$, the first term in the above right-hand side 
is negative and decreasing, while the second is positive and increasing.
We conclude that $k_{D,+}(\xi)$ is a decreasing function that 
maps $(0, \infty)$ into $(0,-1)$.
\medskip

On the other hand, $k_{N,-}$ is a smooth function on $\R_+$, $\lim_{\xi \to 0^+} k_{N,-}(\xi) = - \infty$, $\lim_{\xi \to +\infty} k_{N,-}(\xi) = -1$, and 
it holds:
\begin{eqnarray*}
k_{N,-}(\xi) &=&
\left( -\ds\frac{\cosh[\xi(\pi - \a)] + 1}{\sinh[\xi(\pi - a)]} \right)
\tanh^{-1}(\xi \a/2),
\end{eqnarray*}
As functions of $\xi > 0$, the first term in the above right-hand side 
is negative and increasing, while the second is positive and decreasing.
It follows that $k_{N,-}$ is a strictly increasing function of $\xi$ that maps
$(0,\infty)$ into $(-\infty,-1)$.
\medskip

Thus, for any $-1 < k < 0$ (resp. $-\infty < k < -1$)
there exists a unique $\xi$ such that $k = k_{D,+}(\xi)$
(resp. $k = k_{N,-}(\xi))$.
We also note that $k_{D,\pm}$ and $k_{N,\pm}$ are even functions of $\xi \in \R$,
so that if $u = r^{i\xi} \vf(\theta)$ is a singular solution,
so is $r^{-i\xi} \vf(\theta)$.

\medskip 
We summarize our findings in a technical lemma:
\begin{lemma}\label{lem.singsol}
For any $k < 0$, $k\neq -1$, there exists $\xi >0$ and a $2\pi$-periodic function $\varphi \in H^1_\#(0,2\pi)$ such that the function $u$ defined by
\begin{equation}\label{eq.formu}
 u(x_1,x_2) = \Re(r^{i\xi} \varphi(\theta)) \in L^\infty(\Omega)
 \end{equation}
is a solution of (\ref{eq_cond0}) in the sense of distributions.
In addition, the function $\varphi$ in (\ref{eq.formu}) solves 
$$ (a(\theta)\vf^{\prime}(\theta))^\prime - \xi^2 a(\theta) \vf(\theta) = 0.$$
\end{lemma} 

\begin{remark}
\noindent \begin{itemize} 
\item One can check that
$k_{N,-}(\xi) < k_{N,+}(\xi) < -1 < k_{D,-}(\xi) < k_{D,+}(\xi) < 0$ for all $\xi >0$.
\item In the case where $k > 0$, the same procedure yields solutions of~(\ref{eq_cond0})
of the form $u(r,\theta) = r^\xi \vf(\theta)$ for some $0 < \xi < 1$ and $\varphi \in H^1_\#(0,2\pi)$;  
such functions are in $H^1_{\text{\rm loc}}(\R^2) \setminus H^2_{\text{\rm loc}}(\R^2)$.
\end{itemize}
\end{remark}


\section{Characterization of the spectrum of $T_D$} \label{sec_4}

In this section, we now proceed to the identification of the spectrum of $T_D$.
\begin{theorem}\label{th.specbowtie}
The operator $T_D$ has only essential spectrum and
\begin{eqnarray*}
\sigma(T_D) &=& [0,1].
\end{eqnarray*}
\end{theorem}
 
\begin{proof}
Using Proposition~\ref{prop_TD} and the fact that $\sigma_{\text{\rm ess}}(T_D)$ is closed, 
it is enough to show that any number $\beta \in (0,1)$, $\beta \neq \frac{1}{2}$ lies in the
essential spectrum of $T_D$.
The proof relies on the same ingredients as that of Theorem~2 in~\cite{BZ} and
we reproduce it for the sake of completeness.
\medskip

\textit{Step 1:} Using the singular solutions $u$ to the transmission problem (\ref{eq_cond0}) (see Lemma \ref{lem.singsol})
calculated in the previous section, we aim at constructing a singular Weyl sequence for the operator $T_D$ and the value $\beta$, namely, 
a sequence of functions $u_\e \in H^1_0(\Omega)$ satisfying the following properties (see Section \ref{sec.weyl}):
\begin{equation} \label{def_singW}
\left\{ \begin{array}{ccll}
||u_\e||_{H^1_0(\Omega)} &=& 1,
\\
(\beta \Id- T_D)u_\e &\to& 0 & \textrm{strongly in}\; H^1_0(\Omega),
\\
u_\e &\to& 0 & \textrm{weakly in}\; H^1_0(\Omega).
\end{array} \right.
\end{equation}
\medskip

To this end, let $\rho < \frac{r_0}{2}$; we introduce two smooth cut-off functions $\chi_1, \chi_2 : \R^+ \rightarrow [0,1]$ 
such that for some constant $C > 0$, the following relations hold:
\begin{equation} \label{def_chi}
 \begin{array}{lll}
\chi_1(s) = 0 \text{ for } |s| \leq 1, & \chi_1(s) = 1  \text{ for } |s| \geq 2, & |\chi_1^\prime(s)| \leq C \text{ for } s \geq 0, \\ 
 \chi_2(s) = 1  \text{ for } |s| \leq \rho,& \chi_2(s) = 0  \text{ for } |s| \geq 2\rho,   &  |\chi_2^\prime(s)| \leq C  \text{ for } s \geq 0.
\end{array}
\end{equation}

For $\e>0$ small enough, we set $\chi_1^\e(r) = \chi_1(\frac{r}{\e})$, and define 
\begin{eqnarray} \label{def_ue}
u_\e(x) &=& s_\e \chi^\e_1(r) \chi_2(r) u(x), \quad x \in \Omega,
\end{eqnarray}
where the normalization constant $s_\varepsilon$ is chosen so that $\lvert\lvert u_\e \lvert\lvert_{H^1_0(\Omega)} = 1$.
\medskip 

\textit{Step 2:} We estimate the constant $s_\e$. To this end, we decompose
\begin{equation}\label{eq.decompnue}
\int_\Omega{\lvert \nabla u_\varepsilon \lvert^2 \:dx} = s_\varepsilon^2 (J_{1,\varepsilon} + m_\e + J_2), 
 \end{equation}
where
$$ J_{1,\e} = \int_{B_{2\e} \setminus \overline{B_\e}}{\lvert\nabla u_\varepsilon \lvert^2 \:dx} = \int_{B_{2\e} \setminus \overline{B_\e}}{\lvert u \nabla \chi_1^\e + \chi_1^\e\nabla u \lvert^2\:dx},$$
$$ m_\varepsilon = \int_{B_\rho \setminus \overline{B_{2\e}}}{\lvert \nabla u_\varepsilon \lvert^2 \:dx} =  \int_{B_\rho \setminus \overline{B_{2\e}}}{\lvert \nabla u \lvert^2 \:dx},$$
and 
$$ J_2 = \int_{B_{2\rho} \setminus \overline{B_{\rho}}}{\lvert \nabla u_\varepsilon \lvert^2 \:dx} =  \int_{B_{2\rho} \setminus \overline{B_{\rho}}}{\lvert \chi_2 \nabla u + u \nabla \chi_2 \lvert^2 \:dx}.$$
Let us first estimate $J_{1,\e}$, using the explicit form (\ref{eq.formu}) for $u$ and a change in polar coordinates:
$$ 
\begin{array}{>{\displaystyle}cc>{\displaystyle}l}
J_{1,\e} &=& \int_{\e}^{2\e}{\int_{0}^{2\pi}{\left(\left\lvert \frac{1}{\varepsilon} \frac{r^{i\xi} + r^{-i\xi}}{2} \varphi(\theta) \chi_1^\prime(\frac{r}{\e}) + i\xi \frac{r^{i\xi} - r^{-i\xi}}{2r} \varphi(\theta) \chi_1(\frac{r}{\e}) \right\lvert^2+\left\lvert \frac{r^{i\xi} + r^{-i\xi}}{2r} \varphi^\prime(\theta) \chi_1(\frac{r}{\e})\right\lvert^2\right) rdr} \:d\theta}\\
&\leq& \frac{C}{\varepsilon^2} \int_{\e}^{2\e}{\int_{0}^{2\pi}{r\: dr}\:d\theta} + C \int_{\e}^{2\e}{\int_{0}^{2\pi}{\frac{1}{r}\: dr}\:d\theta}, \\
&\leq& C.
\end{array}
$$
In the above equation, and throughout the proof, $C$ is a generic constant independent of $\varepsilon$, which may change from one line to the next.\par\medskip

The integral $J_2$ does not depend on $\e$, and since $u$ is smooth on $B_{2\rho} \setminus \overline{B_\rho}$, it is bounded by some constant $C>0$.\par\medskip 

Finally, since $u$ does not belong to $H^1_0(\Omega)$ (recall from (\ref{eq.formu}) that its gradient blows up like $r^{-1}$ as $r \to 0$), it follows that
\begin{equation}\label{eq.metoinfty} 
m_\varepsilon \xrightarrow{\e \to 0} \infty.
\end{equation}
Let us note for further reference that the behavior of $m_\e$ as $\e \to 0$ may be estimated more precisely:
$$ \begin{array}{>{\displaystyle}cc>{\displaystyle}l}
m_\e &=& \int_{2\e}^\rho{\int_{0}^{2\pi}{\left( \xi^2 \left\lvert  \frac{r^{i\xi} - r^{-i\xi}}{2r} \varphi(\theta)\right\lvert^2 + \frac{1}{r^2}\left\lvert  \frac{r^{i\xi} + r^{-i\xi}}{2} \varphi^\prime(\theta)\right\lvert^2 \right)} r dr d\theta},\\
&\leq& C \int_{2\e}^\rho{\int_{0}^{2\pi}{ \frac{1}{r}} dr d\theta},
\end{array}$$ 
and so there exists a constant $C >0$ such that
\begin{equation}\label{eq.estmeps}
m_\e \leq C \lvert \log \e \lvert.
\end{equation}

Recalling (\ref{eq.decompnue}), we obtain
$$ 1 = s_\varepsilon^2 m_\varepsilon (1+ \frac{J_{1,\e} + J_2}{m_\e}), $$
so that there exists a constant $C>0$ such that
\begin{equation}\label{eq.estse}
 \frac{1}{C} m_\e^{-\frac{1}{2}} \leq s_\varepsilon  \leq C m_\e^{-\frac{1}{2}}.
\end{equation}

\textit{Step 3:} We show that $u_\e$ is a Weyl sequence for the operator $T_D$ 
and the value $\beta$. 
To this end, we estimate
$$ \lvert\lvert \beta u_\e - T_D u_\e \lvert\lvert_{H^1_0(\Omega)} = \sup\limits_{v \in H^1_0(\Omega), \atop \lvert\lvert v\lvert\lvert_{H^1_0(\Omega)} =1}{J(v)}, \text{ where }ÊJ(v) := \int_\Omega{\nabla(\beta u_\e - T_D u_\e)\cdot \nabla v \:dx}.$$
Recall from (\ref{eq.eqconducTD}) the alternative expression for $J(v)$
$$
\begin{array}{>{\displaystyle}cc>{\displaystyle}l}
J(v) &=&  \beta \int_{\Omega \setminus \overline{D}}{\nabla u_\e \cdot \nabla v \:dx} + (\beta-1) \int_D{\nabla u_\e \cdot \nabla v \:dx} \\
&=& \beta \int_\Omega{a(x) \nabla u_\e\cdot \nabla v\:dx},
\end{array} $$
with 
$$ a(x) = \left\{Ê
\begin{array}{cl}
1 & \text{if }Êx \in \Omega \setminus \overline{D}, \\
1 - \frac{1}{\beta} & \text{if }Êx \in D.
\end{array}
\right.
$$

Inserting the expression~(\ref{def_ue}) of $u_\e$ in the definition of $J(v)$ yields after elementary calculations:
\begin{eqnarray*}
J(v) &=&
s_\e \beta \ds\int_{\Omega \setminus \overline{D}}{ \nabla u \cdot \nabla (\chi_1^\e \chi_2 v) \:dx}
\;+\; 
s_\e (\beta-1)\ds\int_{D} {\nabla u \cdot \nabla (\chi_1^\e \chi_2 v) \:dx}
\\
&&
\;+\; 
s_\e\beta \ds\int_{\Omega \setminus \overline{D}}{ u \nabla(\chi_1^\e \chi_2) \cdot \nabla v \:dx}
\;+\;
s_\e (\beta -1) \ds\int_{D}{u \nabla(\chi_1^\e\chi_2) \cdot \nabla v \:dx}
\\
&&
\;-\;
s_\e \beta\ds\int_{\Omega \setminus \overline{D}}{ v \nabla u \cdot \nabla(\chi_1^\e\chi_2) \:dx}
\;-\;
s_\e (\beta -1) \ds\int_{D}{v \nabla u \cdot  \nabla(\chi_1^\e\chi_2) \:dx}.
\end{eqnarray*}
Since $u$ satisfies~(\ref{eq_cond0}) and since the test function $\chi_1^\e \chi_2v$ has compact support in $B_\rho \setminus B_\e$, 
the sum of the first two integrals in the right-hand side of the above identity vanishes, so that
\begin{equation}\label{eq.decJv}
 J(v) = \beta s_\e (J_{3,\varepsilon}(v) + J_{4,\varepsilon}(v)),
 \end{equation}
where we have defined: 
\begin{equation}\label{est_J}
 J_{3,\e}(v) = \int_\Omega{au \nabla (\chi_1^\e \chi_2) \cdot \nabla v \:dx} - \int_{B_{2\rho} \setminus \overline{B_\rho}}{av\nabla u \cdot \nabla \chi_2 \:dx},   \text{~ and }ÊJ_{4,\e}(v) = -  \int_{B_{2\e} \setminus \overline{B_\e}}{av\nabla u \cdot \nabla \chi_1^\e \:dx} . 
 \end{equation}
Similar calculations to those involved in the estimate (\ref{eq.estse}) show that
 \begin{equation} \label{est_J3}
|J_{3,\e}(v)| \leq C \, ||v||_{H^1_0(\Omega)}
\left(
\frac{1}{\varepsilon^2}\int_0^{2\e} \ds\int_0^{2\pi}{
|u|^2 |\chi_1^\prime|^2  \,rdrd\theta} +
\int_{\rho}^{2\rho} \ds\int_0^{2\pi}
\left(
|u|^2 |\chi_2^\prime|^2 + |\nabla u|^2 |\chi_2|^2 
\right)\,rdrd\theta
\right),
\end{equation}
and so
\begin{eqnarray}
J_{3,\e}(v) \leq C \lvert\lvert v \lvert\lvert_{H^1_0(\Omega)},
\end{eqnarray}
To estimate the remaining term $J_{4,\e}(v)$, we further decompose
\begin{equation}\label{eq.decompJ4}
J_{4,\e} (v) = \int_{B_{2\e} \setminus \overline{B_\e}}{ 
a \overline{v} \nabla u \cdot \nabla \chi_1^\e \:dx} + \int_{B_{2\e} \setminus \overline{B_\e}}  
{ a (v-\overline{v}) \nabla u \cdot  \nabla \chi_1^\e\:dx},
\end{equation}
where $\overline{v} := \frac{1}{|B_{2\e}|} \int_{B_{2\e}} {v(x) \,dx}$.
The first integral in the above right-hand side reduces to
$$ 
\begin{array}{>{\displaystyle}cc>{\displaystyle}l}
\int_{B_{2\e} \setminus \overline{B_\e}}{ 
a \overline{v} \nabla u \cdot \nabla \chi_1^\e \:dx}  &=&  \frac{\overline{v}}{\e}  \int_{\e}^{2\e}{\int_{0}^{2\pi}{ a(\theta) \chi_1^\prime(\frac{r}{\varepsilon}) \: i\xi  \frac{r^{i\xi} - r^{-i\xi}}{2r} \varphi(\theta) } \:rdrd\theta},\\
&=&  \frac{\overline{v}}{\e}  \int_{0}^{2\pi}{a(\theta) \varphi(\theta) \:d\theta} \: \int_\e^{2\e}{i\xi  \frac{r^{i\xi} - r^{-i\xi}}{2}  \chi^\prime(\frac{r}{\e}) \:dr},  \\
&=& 0,
\end{array}
$$
where we have used the fact that $\varphi \in H^1_\#(0,2\pi)$ is a solution to the equation:
$$(a(\theta)\vf^{\prime}(\theta))^\prime - \xi^2 a(\theta) \vf(\theta) = 0,$$
so that it satisfies $\int_0^{2\pi} a(\theta)  \vf(\theta) \, d \theta = 0$.
Hence, returning to (\ref{eq.decompJ4}), it follows that
\begin{eqnarray*}
|J_{4,\e}(v)| &\leq&
\left(
\ds\int_{B_{2\e} \setminus \overline{B_\e}} a^2 |\nabla u \cdot \nabla \chi_\e|^2 \,dx
\right)^{\frac{1}{2}}
\left(
\ds\int_{B_{2\e}}{ |v - \overline{v}|^2  dx}
\right)^{\frac{1}{2}}.
\end{eqnarray*}
The following Poincar\'e-Wirtinger inequality
$$ \int_{B_{2\e}}{\lvert v - \overline{v} \lvert^2 \:dx} \leq C \varepsilon^2 \int_{B_{2\e}}{\lvert \nabla v \lvert^2 \:dx},$$
where the constant $C$ is independent of $\e$, yields
$$
\begin{array}{>{\displaystyle}cc>{\displaystyle}l}
|J_{4,\e}(v)| &\leq&
C \e ||v||_{H^1_0(\Omega)} 
\left(
\int_{0}^{2\pi} a(\theta)^2 |\vf(\theta)|^2 \,d\theta
\right)^{1/2}
 \quad
\left(\int_{\e}^{2\e}{ \left\lvert  i\xi \ds\frac{r^{i\xi}-r^{-i\xi}}{2r}\right\lvert^2 \frac{1}{\e^2}\chi_1^\prime(\frac{r}{\e})^2
\,r dr} \right)^{1/2} \\
&\leq&
C ||v||_{H^1_0(\Omega)} \left( \ds\int_{\e}^{2\e}{ \frac{dr}{r}} \right)^{1/2}\\
&\leq& C ||v||_{H^1_0(\Omega)}.
\end{array}
$$

We conclude from (\ref{eq.decJv}), (\ref{est_J3}) and the above estimate that
$$ \lvert J(v) \lvert \leq C s_\varepsilon \lvert\lvert v \lvert\lvert_{H^1_0(\Omega)}.$$
Since $s_\e \to 0$ (see (\ref{eq.estse}) and (\ref{eq.metoinfty})), this proves that
$$ \lvert\lvert \beta u_\e - T_D u_\e \lvert\lvert_{H^1_0(\Omega)} \xrightarrow{\e \to 0} 0,$$
and so $u_\e$ is a Weyl sequence for $T_D$ and the value $\beta$.
\medskip

\textit{Step 4:} Finally, we show that $u_\e$ is a \textit{singular} Weyl sequence for $T_D$ 
and the value $\beta$, 
namely that $u_\e \to 0$ weakly in $H^1_0(\Omega)$. Since $u_\e$ has unit norm in $H^1_0(\Omega)$, 
it is enough to prove that $u_\e \to 0$ strongly in $L^2(\Omega)$,
which follows easily from~(\ref{def_ue}), from the boundedness of $\chi_1^\e$, $\chi_2$ and $u$ in $L^\infty(\Omega)$,
and from the fact that $s_\e \to 0$ (viz. (\ref{eq.estse})).

\end{proof}


\section{Comparison with the bowtie with close-to-touching wings} \label{sec_5}

It is interesting to compare the spectral properties of the Poincar\'e variational
operator of $D$ to that of a (Lipschitz) domain $D_\d = D_{1,\d} \cup D_{2,\d}$ 
with only close-to-touching wings. Let us introduce
\begin{eqnarray*}
D_{1,\d} \;=\; (\d/2,0) + D_1,
&\quad &
D_{2,\d} \;=\; (-\d/2,0) + D_2,
\end{eqnarray*}
where the parameter $\d>0$ is sufficiently small so that $D_\d \Subset \Omega$; see Figure \ref{fig.bowtiectt}.

\begin{figure}[!ht]
\centering
\includegraphics[width=0.6\textwidth]{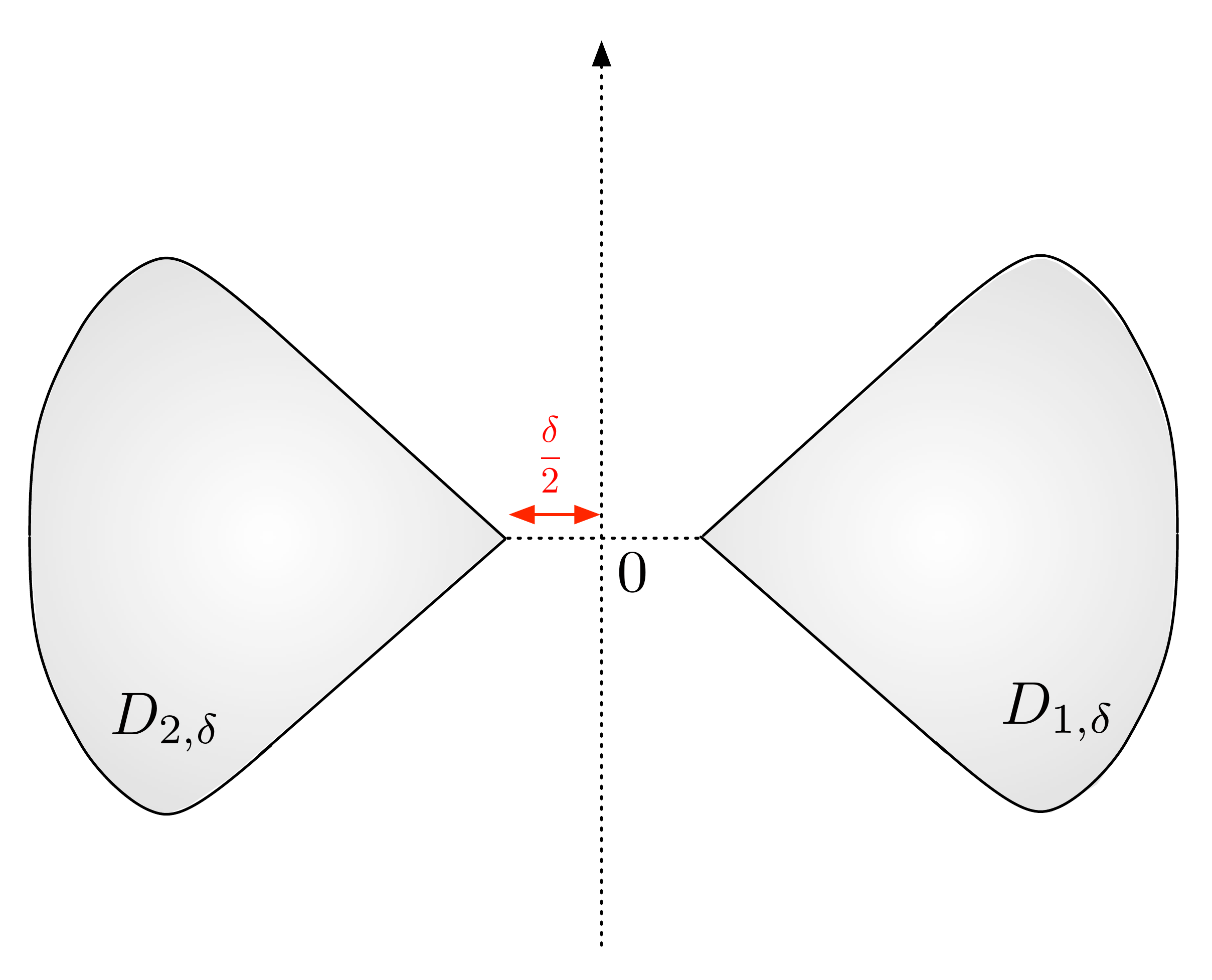}
\caption{\it The bowtie with close-to-touching wings.}
\label{fig.bowtiectt}
\end{figure}

The corresponding Poincar\'e variational operator 
$T_{D_\d}~: H^1_0(\Omega) \longrightarrow H^1_0(\Omega)$ is now defined by
\begin{eqnarray*}
\forall\; v \in H^1_0(\Omega), \quad
\ds\int_\Omega{ \nabla (T_{D_\d}u) \cdot \nabla v \:dx}
&=& \ds\int_{D_\d}{ \nabla u \cdot \nabla v \:dx}.
\end{eqnarray*}

Since $D_\d$ is Lipschitz regular, the study of the spectrum $\sigma(T_{D_\d})$ 
falls into the framework of Sections \ref{sec.sPV0lip} and \ref{sec.NPOlip}, and Proposition \ref{prop_TDd} holds in this case.

More precisely, both domains $D_{1,\delta}$ and $D_{2,\delta}$ have a piecewise smooth 
boundary with a finite number of angles. 
Hence, the results of K.-M. Perfekt and M. Putinar~\cite{PerfektPutinar_2}
apply: the essential spectrum of the associated
Poincar\'e variational operator $T_{D_\delta}$ (and that of the Neumann-Poincar\'e operator ${\mathcal K}_D^*$)
is completely determined by the most acute angle $\alpha$ on the boundary of $D_{1,\delta}$ and $D_{2,\delta}$.
In our context, this takes the form:

\begin{equation*}
\Sess(T_{D_\d}) = \left[\frac{\alpha}{2\pi}, 1-\frac{\alpha}{2\pi}\right]; \:\: \Sess({\mathcal K}^*_{D_\d})  = \left[-\frac{\pi-\alpha}{2\pi}, \frac{\pi-\alpha}{2\pi}\right].
\end{equation*}

Hence, the close-to-touching corners of $D_\d$ are qualitatively less singular than the bowtie feature of $D$, which is associated to an essential spectrum $\sigma(T_D) = [0,1]$. 
A similar phenomenon was already noticed in the article \cite{BV}, investigating
the regularity of solutions to~(\ref{eq_cond}) in the case of the domains $D$ and $D_\d$ for a value $k>0$ of the conductivity. 
In the close-to-touching case, 
the singular part of the solution $u_\d$ to (\ref{eq_cond}) behaves like $r^\eta$ at the vertices, 
with $\eta \geq 2/3$ independently of the value of $k$ and of the angle $\a$.
For the touching case (i.e. in the case of $D$), $u$ behaves also like $r^\eta$ at the contact point,
but $\eta$ can be made as close to $0$ as desired by choosing $k$ sufficiently
close to $0$ or $+\infty$.
\medskip

Our aim is now to shpw that, as $\d \to 0$, the spectrum $\sigma(T_{D_\delta})$
converges to a limiting set which is exactly the spectrum
$\sigma(T_{D}) = [0,1]$ of the limiting physical situation. To this end, we study the \textit{limit spectrum} 
\begin{equation}\label{eq.limspec}
 \lim\limits_{\delta \to 0}{\sigma(T_{D_\d})} := \left\{ \beta \in \mathbb{R}, \:\: \exists \delta_n \downarrow 0, \: \beta_n\in \sigma(T_{D_{\delta_n}}), \:\: \beta_n \to \betaÊ\right\}
 \end{equation}
of the sequence of operators $T_{D_\d}$.

\medskip 
Our analysis relies on the following abstract result for
self-adjoint operators, which is part of the statement of Lemma~(2.8) 
in~\cite{AllaireConca}.

\begin{theorem} \label{thm_AC}
Let $H$ be a Hilbert space and $S_\delta~: H \rightarrow H$ denote a sequence of
self-adjoint operators, with spectrum $\sigma(S_\delta)$. 
Assume that the operators $S_\delta$ converge pointwise
to a limiting operator $S$, with spectrum $\sigma(S)$, in the sense that
\begin{eqnarray} \label{conv_pointwise}
\forall\; u \in H,\quad
\lim_{\delta \to 0} ||S_\delta u - Su || &\rightarrow& 0.
\end{eqnarray}
Then, 
\begin{eqnarray} \label{conv_spectra}
\lim_{\delta \to 0} \sigma(S_\delta) &\supset& \sigma(S),
\end{eqnarray}
where the left-hand set denotes the limit spectrum of the sequence of operators $S_\delta$.
\end{theorem}

\begin{remark}
The statement in \cite{AllaireConca} is more general; in this reference, the result is proved under the additional 
assumption that the operators $S_\delta$ and $S$ are compact, 
but this hypothesis is not necessary for the version presented in Theorem \ref{thm_AC}.
\end{remark}
%
%

We now prove

\begin{proposition}\label{prop.cvstrong}
The operators $T_{D_\delta}$ converge pointwise to $T_D$ as $\delta \to 0$,
in the sense that
$$\forall\; u \in H^1_0(\Omega), \quad
\lim_{\delta \to 0}||T_{D_\delta} u - T_D u||_{H^1_0(\Omega)} = 0.
$$
\end{proposition}

{\bf Proof:}
Fix $u \in H^1_0(\Omega)$ and consider 
\begin{eqnarray*}
||T_{D_\delta} u - T_D u||_{H^1_0(\Omega)}^2
&=&
\ds\int_\Omega |\nabla T_{D_\delta} {u - \nabla T_D u|^2 \:dx}
\\
&=&
\ds\int_{D_\delta} \nabla u \cdot 
\nabla \left( T_{D_\delta} u - T_D u \right) \:dx
-
\ds\int_{D} \nabla u \cdot 
\nabla \left( T_{D_\delta} u - T_D u \right) \:dx
\\
&=& 
\ds\int_\Omega (\mathds{1}_{D_\delta} - \mathds{1}_D) \nabla u \cdot
\nabla \left( T_{D_\delta} u - T_D u \right) \:dx
\\
&\leq&
\left( \ds\int_\Omega (\mathds{1}_{D_\delta} - \mathds{1}_D) |\nabla u|^2  \:dx \right)^{1/2}
||T_{D_\delta} u - T_D u||_{H^1_0(\Omega)}.
\end{eqnarray*}
The Lebesgue Dominated Convergence Theorem shows that the first integral 
on the right-hand side tends to $0$ as $\d \to 0$, which proves the Proposition.
\finproof
\medskip

Combining Proposition \ref{prop.cvstrong}, Theorem~\ref{thm_AC} and the fact that the spectrum of each $T_{D_\d}$ is contained in $[0,1]$ (see Proposition \ref{prop_TDd}), 
we obtain:

\begin{corollary}\label{cor.limspec} 
The limiting spectrum (\ref{eq.limspec}) of the operators $T_{D_\d}$ is exactly that of the Poincar\'e variational operator of the bowtie antenna $D$:
\begin{equation*}
\lim_{\delta \to 0} \sigma(T_{D_\delta}) = \sigma(T_D) = [0,1].
\end{equation*}
\end{corollary}

This result deserves a few additionnal comments. 
As we have mentionned, the essential spectrum of $T_{D_\d}$ is exactly the interval $ [\ds\frac{\a}{2\pi}, 1 - \ds\frac{\a}{2\pi}]$
independently of $\delta$, whereas the above corollary shows that in the limit $\delta \to 0$, the spectrum $\sigma(T_{D_\d})$ 
must densify so as to occupy the whole interval $[0,1]$. The only possible way for this to happen is that for $\d$ sufficiently 
small $T_{D_\d}$ must develop eigenvalues in the intervals $[0,\ds\frac{\alpha}{2\pi})$
and $(1-\ds\frac{\a}{2\pi},1]$, which become denser as $\delta \to 0$.
Let us point out that such a densification phenomenon has been observed in different physical contexts; see ~\cite{HelsingKangLim,HelsingMcPhedranMilton} and~\cite{BonnetierDapognyHelsingKang}.


\section{Another approach to the limit spectrum of bowties 
with close-to-touching wings} \label{sec_6}

The purpose of this section is to provide an alternative proof of the fact that 
$\sigma(T_{D_\d})$ contains eigenvalues if the distance
between the wings is sufficiently small.
This fact is indeed contained in Corollary \ref{cor.limspec}, 
but the forthcoming proof is more direct, 
and sheds light on the behavior of the eigenfunctions of $T_{D_\d}$. 
The main result of this section is the following:

\begin{theorem}
For $\d > 0$ small enough, the operator $T_{D_\d}$ has eigenvalues in the range 
$\left(1 - \frac{\a}{2\pi}, 1\right)$ and in the range $\left(0,\frac{\a}{2\pi}\right)$, i.e., 
outside the essential spectrum $\Sess(T_{D_\d})$.
\end{theorem}
\begin{proof}
Recalling the orthogonal decomposition (\ref{eq.decHd}), 
let us denote by $\beta_{\d}^-$ and $\beta_\d^+$ the lower and upper bounds of the spectrum of $T_{D_\d}$ deprived of the trivial eigenvalues $0$ and $1$, i.e.
$$ \beta_\d^- = \inf_{\sigma(T_{D_\d}) \setminus \left\{Ê0,1\right\}} \text{ and }  \beta_\d^+ = \sup_{\sigma(T_{D_\d}) \setminus \left\{Ê0,1\right\}}$$ 
Relying on a spectral representation for the operator $T_{D_\d} : {\mathcal H}_{D_\d} \to {\mathcal H}_{D_\d}$ (see e.g. \cite{ReedSimon}), these bounds are given by the Rayleigh quotients:
\begin{equation}\label{eq.varprincip}
 \beta_\d^- = \min\limits_{w \in H^1_0(\Omega) \atop w \perp \Ker(T_{D_\d})}{\frac{\displaystyle{\int_{D_\d}{\lvert \nabla w \lvert^2 \:dx}}  }{\displaystyle{\int_{\Omega}{\lvert \nabla w \lvert^2 \:dx}}Ê}} 
\;\text{ and }\;Ê\beta_\d^+ = \max\limits_{w \in H^1_0(\Omega) \atop w \perp \Ker(\Id-T_{D_\d})}{\frac{\displaystyle{\int_{D_\d}{\lvert \nabla w \lvert^2 \:dx}}  }{\displaystyle{\int_{\Omega}{\lvert \nabla w \lvert^2 \:dx}}Ê}}.
 \end{equation}
\medskip

Let us now pick a value 
$\beta \notin [\ds\frac{\a}{\pi},1- \ds\frac{\a}{\pi}]$, so that $\beta$ lies outside the essential spectrum $\sigma_{\textrm{ess}}(T_{D_\d})$ for any $\d > 0$.
Our aim is to prove that there exists a sequence of functions $Z_\d \in H^1_0(\Omega)$ which is orthogonal to $\Ker(T_{D_\d})$ (resp. to $\Ker(\Id-T_{D_\d})$) such that:
$$\beta =
\lim_{\d \to 0}
\frac{\displaystyle{\int_{D_\d}{ |\nabla Z_\d|^2\:dx}}}
{\displaystyle{\int_\Omega { |\nabla Z_\d|^2\:dx}}}. $$

\medskip

Let $k = 1 - \ds\frac{1}{\beta}$ be the conductivity associated to $\beta$ (see Section \ref{sec.TD}). 
We take on the construction of $u_\e$ carried out in Section \ref{sec_4}: 
let $u$ denote the function supplied by Lemma \ref{lem.singsol}:
\begin{eqnarray} \label{def_u2}
u(x) &=& \Re(r^{i \xi}) \vf(\theta),
\end{eqnarray}
where $\xi$ satisfies 
\begin{eqnarray*}
d_D(\xi) = 0 &\quad \textrm{or}\quad& d_N(\xi) = 0,
\end{eqnarray*}
according to~(\ref{det_D}) and (\ref{det_N}).
\medskip

Let $0 < \rho$ be sufficiently small, and let $\chi_1$, $\chi_2$ be the cut-off functions 
defined as in~(\ref{def_chi}); for $0<\e<\rho$, we define:
\begin{eqnarray*}
u_\e(x) &=& s_\e \chi_1(\frac{r}{\e})\chi_2(r) u(x).
\end{eqnarray*}
As in~(\ref{def_ue}), the normalization constant $s_\e$ is chosen so that 
$||u_\e||_{H^1_0(\Omega)} = 1$.
Recall from (\ref{eq.estmeps}) that there exists a constant $C >0$ such that:
\begin{eqnarray} \label{estim_se2}
s_\e \leq C\ds\frac{1}{|\log(\e)|^{\frac{1}{2}}}.
\end{eqnarray}
The calculations performed in Section~\ref{sec_4} have revealed that the sequence $u_\e$ satisfies
\begin{eqnarray} \label{estim_ue}
\lim_{\e \to 0}||(\beta I - T_D)u_\e||_{H^1_0(\Omega)} &=& 0.
\end{eqnarray}
Recalling (\ref{def_T}), this implies in particular that
\begin{eqnarray} \label{lim_energ_e}
\beta &=& \lim_{\e \to 0}
\frac{\displaystyle{\int_{D}{ |\nabla u_\e|^2 \:dx}}}
{\displaystyle{\int_\Omega { |\nabla u_\e|^2 \:dx}}}
\;=\;  \lim_{\e \to 0}\int_{D}{ |\nabla u_\e|^2 \:dx}.
\end{eqnarray}
\medskip

Let us next turn to the configuration $D_\d$; for a small parameter $\varepsilon >0$ to be specified later, we
define a function $v_{\d,\e}$ by:
\begin{eqnarray} \label{def_vde}
v_{\d,\e}(x_1,x_2) &=&
\left\{ \begin{array}{cl}
u_\e(x_1+\frac{\d}{2},x_2) & \text{if } x_1 < -\frac{\d}{2},
\\
u_\e(x_1-\frac{\d}{2},x_2) &\text{if } x_1 > \frac{\d}{2},
\\
u_\e(0,x_2) & \textrm{otherwise}.
\end{array} \right.
\end{eqnarray}
Note that, by construction, $v_{\d,\e} \in H^1_0(\Omega)$ and:
\begin{eqnarray} \label{estim_vde1}
\ds\int_{D_\d}{ |\nabla v_{\d,\e}|^2 \:dx} &=&
\ds\int_{D} {|\nabla u_\e|^2 \:dx}.
\end{eqnarray}
Additionnally, in view of~(\ref{def_u2}), we have

\begin{equation} \label{estim_vde2}
\begin{array}{>{\displaystyle}cc>{\displaystyle}l}
\int_\Omega { |\nabla v_{\d,\e}|^2 \:dx} &=&
\int_{x_1 < -\frac{\d}{2}} {|\nabla u_\e(x_1 + \frac{\d}{2},x_2)|^2 dx}+ \int_{x_1 >\frac{\d}{2}} {|\nabla u_\e(x_1 -  \frac{\d}{2},x_2)|^2 \:dx} +
\int_{|x_1| < \frac{\d}{2}} {|\partial_{x_2} u_\e(0,x_2)|^2 \:dx} \\ 
&=&
\int_\Omega { |\nabla u_\e|^2 \:dx}
\;+\;
s_\e^2 \int_{|x_1| < \frac{\d}{2}}  {
|\partial_{x_2} \left[
\chi_1(\frac{x_2}{\e}) \chi_2(x_2) u(0,x_2)
\right]|^2 \:dx}.
\end{array}
\end{equation}
We now estimate the last integral in the above expression; to this end,

\begin{equation} \label{estim_vde3}
\begin{array}{>{\displaystyle}cc>{\displaystyle}l}
 \int_{|x_1| < \frac{\d}{2}}  {
\left \lvert \partial_{x_2} (
\chi_1(\frac{x_2}{\e}) \chi_2(x_2) u(0,x_2)
) \right\lvert^2 \:dx} &\leq& \frac{\delta}{\e^2} \int_\e^{2\e}{\left\lvert \chi_1^\prime(\frac{x_2}{\e}) \chi_2(x_2) \cos(\xi\log\lvert x_2\lvert) \right\lvert^2 \:dx_2} \\
&& + \delta \int_\rho^{2\rho}{\left\lvert \chi_1(\frac{x_2}{\e}) \chi_2^\prime(x_2) \cos(\xi\log\lvert x_2\lvert) \right\lvert^2 \:dx_2} \\
&& +\d \int_\e^{2\rho}{ \frac{\xi^2}{x_2^2} \left\lvert \chi_1(\frac{x_2}{\e}) \chi_2(x_2) \sin(\xi\log\lvert x_2\lvert) \right\lvert^2 \:dx_2}.\\
&\leq& C \frac{\d}{\e},
\end{array}
\end{equation}
where the constant $C>0$ is independent of $\d$ and $\e$.
Combining (\ref{estim_vde1}), (\ref{estim_vde2}) and (\ref{estim_vde3}), we find that
\begin{eqnarray*}
\frac{\displaystyle{\int_{D_\d} |\nabla v_{\d,\e}|^2 \:dx}}
{\displaystyle{\int_\Omega {|\nabla v_{\d,\e}|^2 \:dx}}}
&=&
\ds\frac{\displaystyle{\int_{D}{ |\nabla u_\e|^2\:dx}}}
{\displaystyle{\int_\Omega{ |\nabla u_\e|^2 \:dx}} + \frac{s_\e^2 \d}{\e} B_{\e,\d} },
\end{eqnarray*}
where $B_{\e,\d}$ is uniformly bounded with respect to $\e$ and $\d$.
Finally, choosing $\e = \d$ and using (\ref{lim_energ_e}) and~(\ref{estim_se2}), 
it follows that the function $w_\d := v_{\d,\d}$ satisfies
\begin{eqnarray} \label{estim_wd_spec}
\left|
\beta - \ds\frac{\displaystyle{\int_{D_\d} { |\nabla w_\d|^2 \:dx}}}
{\displaystyle{\int_\Omega { |\nabla w_\d|^2 \:dx}}}
\right|
&\leq& \ds\frac{C}{|\log\d|} \to 0,
\;\textrm{as}\; \d \to 0.
\end{eqnarray}\par
\medskip

On a different note, it will be useful for further purpose to notice that $w_\d$ is somehow `close' to $u_\d$. 
More precisely, the following result will come in handy:

\begin{lemma}\label{lem.cvwdud}
The following convergence holds: 
$$ \lvert\lvert u_\d - w_\d \lvert\lvert_{H^1_0(\Omega)} \to 0 \text{ as } \d \to 0.$$
\end{lemma} 
The proof of Lemma \ref{lem.cvwdud} is technical and is postponed to the end of this section.

\medskip

To summarize: we have constructed a series of `test' functions $w_\d \in H^1_0(\Omega)$ 
whose energy ratio converges to the desired value $\beta$ as $\d \to 0$. 
To use these functions in the variational principles (\ref{eq.varprincip}), 
we now construct from $w_\d$ a new series of functions $Z_\d \in H^1_0(\Omega)$ which satisfy the orthogonality conditions $Z_\d \perp \Ker(T_{D_\d})$ or $Z_\d \perp \Ker(\Id-T_{D_\d})$.  
To achieve this, we separate both cases.
\medskip 

\textit{Case 1: $1- \frac{\a}{\pi} < \beta < 1$.}

Let $W_\d$ denote the orthogonal projection of $w_\d$ on $\Ker(\Id - T_{D_\d}) = H^1_0(D_\d)$ and let
$Z_\d = w_\d - W_\d$. 
We also define the function:
\begin{eqnarray*}
U_\d(x) &=& \mathds{1}_{\{x_1 < 0\}}(x) W_\d(x_1 - \frac{\d}{2},x_2) + \mathds{1}_{\{x_1 > 0\}}(x) W_\d(x_1 + \frac{\d}{2},x_2).
\end{eqnarray*}
Obviously, $\lvert\lvert U_\d \lvert\lvert_{H^1_0(\Omega)} = \lvert\lvert W_\d \lvert\lvert_{H^1_0(\Omega)}$. 
Also, since $W_\d \in H^1_0(D_\d)$, there exists a sequence of smooth functions
$(W_{n,\d})_{n \geq 1}$ with compact support inside $D_\d$ such that $W_{n,\d} \to W_\d$ strongly 
in $H^1_0(\Omega)$. It is then easy to check that the functions
\begin{eqnarray*}
U_{n,\d}(x) &:=& \mathds{1}_{\{x_1 < 0\}}(x) W_{n,\d}(x_1 - \frac{\d}{2},x_2) + 
\mathds{1}_{\{x_1 > 0\}}(x) W_{n,\d}(x_1 + \frac{\d}{2},x_2)
\end{eqnarray*}
are smooth with compact support inside $D$ and that they satisfy $U_{n,\d} \to U_\d$ strongly in $H^1_0(\Omega)$.
It follows that $U_\d \in H^1_0(D)$.\par
\medskip

Now, at first using (\ref{estim_vde2}), (\ref{estim_vde3}) and the orthogonality of $W_\d$ and $Z_\d$ yields: 
$$ 
\begin{array}{>{\displaystyle}cc>{\displaystyle}l}
1+ o(1) = \int_\Omega{\lvert \nabla w_\d \lvert^2 \:dx}  &=&  \int_\Omega{\lvert \nabla W_\d \lvert^2 \:dx}  + \int_\Omega{\lvert \nabla Z_\d \lvert^2 \:dx} \\
 &=&  \int_{D_\d}{\lvert \nabla W_\d \lvert^2 \:dx}  + \int_\Omega{\lvert \nabla Z_\d \lvert^2 \:dx}, 
\end{array}
$$
where $o(1) \to 0$ as $\delta \to 0$. Also, from (\ref{estim_wd_spec}), using again (\ref{estim_vde2}) and (\ref{estim_vde3}), we infer:
 $$ 
\begin{array}{>{\displaystyle}cc>{\displaystyle}l}
\beta+ o(1) = \int_{D_\d}{\lvert \nabla w_\d \lvert^2 \:dx}  &=&  \int_{D_\d}{\lvert \nabla W_\d \lvert^2 \:dx}  + \int_{D_\d}{\lvert \nabla Z_\d \lvert^2 \:dx} + 2 \int_{D_\d}{\nabla W_\d \cdot \nabla Z_\d \:dx}\\
 &=&   \int_{D_\d}{\lvert \nabla W_\d \lvert^2 \:dx}  + \int_{D_\d}{\lvert \nabla Z_\d \lvert^2 \:dx}, \\
\end{array}
$$
since 
$$ 
\int_{D_\d}{\nabla W_\d \cdot \nabla Z_\d \:dx} = \int_{\Omega}{\nabla (T_{D_\d} W_\d) \cdot \nabla Z_\d  \:dx} 
= \int_\Omega{\nabla W_\delta \cdot \nabla Z_\d \:dx} = 0.  
$$
Hence, our purpose is now to prove that $\lvert\lvert W_\d \lvert\lvert_{H^1_0(\Omega)} \to 0$ as $\d \to 0$. 

\medskip
To this end, we first observe that, on the one hand, since $W_\d \in \Ker(\Id - T_{D_\d})$, 
\begin{equation} \label{eq.estintWw1}
\begin{array}{>{\displaystyle}cc>{\displaystyle}l}
\int_{\Omega}{\nabla((T_{D_\d} - \beta \Id)w_\d) \cdot \nabla W_\d \:dx} &=& \int_{\Omega}{\nabla w_\d \cdot \nabla ((T_{D_\d} - \beta \Id)W_\d) \:dx}, \\ 
&=& (1-\beta) \int_{\Omega}{\nabla w_\d \cdot \nabla W_\d \:dx} , \\
&=& (1-\beta) \lvert\lvert W_\d \lvert\lvert^2_{H^1_0(\Omega)}.
\end{array}
\end{equation}

On the other hand, recalling~(\ref{def_vde}) with $\e = \d$, a change of variables yields:
\begin{equation}\label{eq.estTDdwW}
\begin{array}{>{\displaystyle}cc>{\displaystyle}l}
\int_\Omega{\nabla(T_{D_\d}w_\d) \cdot \nabla W_\d \:dx} &=&
\ds\int_{D_\d} {\nabla w_\d \cdot \nabla W_\d \:dx}
\\
&=& 
\ds\int_{D} {\nabla u_\d \cdot \nabla U_\d \:dx}
\;=\;  \int_\Omega{\nabla (T_D u_\d) \cdot  \nabla U_\d \:dx },
\end{array}
\end{equation}
and also, since $W_\d$ and $U_\d$ are supported in $D_\d$ and in $D$ respectively,
\begin{eqnarray}\label{eq.estwW}
\int_\Omega{\nabla w_\d \cdot \nabla W_\d \:dx} &=& \ds\int_{D_\d} {\nabla w_\d \cdot \nabla W_\d \:dx}
\;=\; \int_\Omega{\nabla u_\d \cdot \nabla U_\d \:dx }.
\end{eqnarray}

Combining (\ref{eq.estTDdwW}) and (\ref{eq.estwW}) thus implies:
\begin{eqnarray*}
\int_{\Omega}{\nabla((T_{D_\d} - \beta \Id)w_\d) \cdot \nabla W_\d \:dx} 
&=& \int_{\Omega}{\nabla((T_{D} - \beta \Id)u_\d) \cdot \nabla U_\d \:dx}  
\\
&\leq& ||(T_D - \beta \Id)  u_\d||_{H^1_0(\Omega)} ||U_\d||_{H^1_0(\Omega)}
\\
&=& ||(T_D - \beta \Id)  u_\d||_{H^1_0(\Omega)} ||W_\d||_{H^1_0(\Omega)}.
\end{eqnarray*}
Combining this estimate with~(\ref{eq.estintWw1}), and in view of~(\ref{estim_ue}), 
we obtain
\begin{eqnarray*}
(1- \beta) ||W_\d||_{H^1_0(\Omega)} &\leq&
||(T_D  - \beta \Id) u_\d||_{H^1_0(\Omega)} \;=\; o(1)
\quad \textrm{as}\; \d \to 0.
\end{eqnarray*}
Since $\beta \neq 1$, we conclude that $||W_\d||_{H^1_0(\Omega)} \to 0$, as expected. 

\medskip 
This together with (\ref{estim_wd_spec}) finally implies:
\begin{eqnarray*}
\beta = \lim_{\d \to 0}
\frac{\displaystyle{\int_{D_\d}{ |\nabla w_\d|^2\:dx}}}
{\displaystyle{\int_\Omega { |\nabla w_\d|^2\:dx}}}
&=&
\lim_{\d \to 0}
\frac{\displaystyle{\int_{D_\d}{ |\nabla Z_\d|^2\:dx}}}
{\displaystyle{\int_\Omega { |\nabla Z_\d|^2\:dx}}},
\end{eqnarray*}
and so, since $Z_\d \perp \Ker(\Id - T_{D_\d})$:
\begin{eqnarray} \label{estim_d_eigenv_p}
\beta_\d^+ = \max\limits_{
w \in H^1_0(\Omega)
\atop
w \perp \Ker(\Id - T_{D_\d})} 
\frac{\displaystyle{\int_{D_\d}{ |\nabla w|^2\:dx}}}
{\displaystyle{\int_\Omega{ |\nabla w|^2\:dx}}} 
&\geq& \beta + o(1),
\end{eqnarray}
which is the desired result.
\medskip

\textit{Case 2: $0 < \beta < \frac{\a}{\pi}$.}\par\medskip

Recalling~(\ref{estim_wd_spec}), we again decompose~$w_\d = W_\d + Z_\d$,
where $W_\d$ now denotes the orthogonal projection of $w_\d$ on $\Ker(T_{D_\d})$,
so that in particular $\nabla W_\d = 0$ inside $D_\d$.
Again, our aim is to prove that $W_\d \to 0$ strongly in $H^1_0(\Omega)$ as $\d \to 0$. 
\medskip

This follows from the chain of inequalities:
\begin{equation*}
\begin{array}{>{\displaystyle}cc>{\displaystyle}l}
||W_\d||^2_{H^1_0(\Omega)} 
&=&  \ds\int_\Omega {\nabla W_\d \cdot \nabla(w_\d - Z_\d)  \:dx}
\\
&=& \ds\int_\Omega {\nabla W_\d \cdot \nabla w_\d \:dx}
\\
&=&
\ds\frac{1}{\beta}\,\int_\Omega { \nabla ((\beta I - T_{D_\d})u_\d) \cdot \nabla W_\d \:dx}
+ \int_{\Omega}{\nabla(w_\d - u_\d) \cdot \nabla W_\d \:dx} +  \ds\frac{1}{\beta}\,\int_\Omega { \nabla T_{D_\d} u_\d \cdot \nabla W_\d \:dx}
\\
&\leq& 
\ds\frac{1}{\beta}\, ||(\beta I - T_{D_\d})u_\d||_{H^1_0(\Omega)} ||W_\d||_{H^1_0(\Omega)} + \lvert\lvert u_\d - w_\d \lvert\lvert_{H^1_0(\Omega)} \lvert\lvert W_\d \lvert\lvert_{H^1_0(\Omega)}
+ \ds\frac{1}{\beta} \left| \ds\int_{D_\d}{ \nabla w_\d \cdot \nabla W_\d  \:dx} \right|,
\end{array}
\end{equation*}
and so: 
$$ \lvert\lvert W_\d \lvert\lvert_{H^1_0(\Omega)} \leq \ds\frac{1}{\beta}\, ||(\beta I - T_{D_\d})u_\d||_{H^1_0(\Omega)}  + \lvert\lvert u_\d - w_\d \lvert\lvert_{H^1_0(\Omega)}.$$
It thus follows from~(\ref{estim_ue}) and Lemma \ref{lem.cvwdud} that $||W_\d||_{H^1_0(\Omega)} \to 0$, 
so that
\begin{eqnarray*}
\lim_{\d \to 0}
\ds\frac{\displaystyle{\int_{D_\d} {|\nabla w_\d|^2 \:dx}}}
{\displaystyle{\int_\Omega {|\nabla w_\d|^2 \:dx}}}
&=&
\lim_{\d \to 0}
\ds\frac{\displaystyle{\int_{D_\d}{ |\nabla Z_\d|^2\:dx}}}
{\displaystyle{s\int_\Omega{ |\nabla Z_\d|^2\:dx}}}
\;=\; \beta,
\end{eqnarray*}
which yields, since $Z_\d \perp \Ker(T_{D_\d})$,  
\begin{eqnarray} \label{estim_d_eigenv_m}
\min_{
w \in H^1_0(\Omega)
\atop
w \perp \Ker(T_{D_\d})}
\frac{\displaystyle{\int_{D_\d}{ |\nabla w|^2 \:dx}}}
{\displaystyle{\int_\Omega{ |\nabla w|^2\:dx}}} &\leq& \beta + o(1).
\end{eqnarray}
\medskip

We conclude from~(\ref{estim_d_eigenv_p}) and~(\ref{estim_d_eigenv_m})
that for $\d > 0$ small enough, $T_{D_\d}$ necessarily has eigenvalues in the range 
$[1 - \frac{\a}{\pi}, 1)$ and in the range $(0,\frac{\a}{\pi})$, i.e., 
outside the essential spectrum.
\end{proof}
 
We eventually prove the missing link in the above discussion.

\begin{proof}[Proof of Lemma \ref{lem.cvwdud}]
By definition, $u_\d$ has compact support inside $B_{2\rho}$, while $w_\d$ has compact support in the stadium
$$ S_\d := B_{2\rho}(-\frac{\delta}{2},0) \cup  L_\d \cup B_{2\rho}(\frac{\delta}{2},0), \text{ where }ÊL_\d := \left\{ x= (x_1,x_2) \in \Omega, \:\: \lvert x_1 \lvert < \frac{\d}{2}, \: \lvert x_2 \lvert < 2_\rho \right\}.$$
Denote
$$H_\delta^- = \left\{ x\in B_{2\rho} \setminus \overline{L_\d}, \:\: x_1 < 0 \right\}, \text{ and } H_\delta^+ = \left\{ x\in B_{2\rho} \setminus \overline{L_\d}, \:\: x_1 >  0 \right\}.$$ 
Using that $\lvert S_\d \setminus \overline{B_{2\rho}} \lvert \to 0$ as $\d \to 0$, and the uniform boundedness of $u_\d$ and $w_\d$ `far' from $0$, one has first: 
$$ 
\begin{array}{>{\displaystyle}cc>{\displaystyle}l}
 \lvert\lvert u_\d - w_\d \lvert\lvert_{H^1_0(\Omega)}^2  &=& \int_{S_\d}{ \lvert \nabla u_\d - \nabla w_\d \lvert^2 \:dx}, \\
 & =& \int_{L_\d}{ \lvert \nabla u_\d - \nabla w_\d \lvert^2 \:dx} +  \int_{B_{2\rho} \setminus \overline{L_\d}}{ \lvert \nabla u_\d - \nabla w_\d \lvert^2 \:dx} + o(1),\\
 &=:&  I_\d^- + I_\d^+ + I_\d^L + o(1),
\end{array} 
$$
where we have introduced the following three integrals (recalling the definition (\ref{def_vde}) of $w_\d$):
$$ I_\d^- :=  \int_{H_\delta^-}{ \lvert \nabla u_\d(x_1,x_2) - \nabla u_\d(x_1+\frac{\d}{2},x_2) \lvert^2 \:dx} , \:\: I_\d^+ :=  \int_{ H_\delta^+}{\lvert \nabla u_\d(x_1,x_2) - \nabla u_\d(x_1-\frac{\d}{2},x_2) \lvert^2 \:dx},$$
$$ I_\d^L := \int_{L_\d}{ \lvert \nabla u_\d(x_1,x_2) - \nabla u_\d(0,x_2) \lvert^2 \:dx},$$
We now prove that $I_\d^-$, $I_\d^+$ and $I_\d^L$ vanish as $\d \to 0$. \par
\medskip
$\bullet$ \textit{Proof of the convergence $I_\d^- \to 0$:}
A simple calculation yields: 
$$ 
\begin{array}{>{\displaystyle}cc>{\displaystyle}l}
I_\d^- &=&  \int_{H_\delta^- \cap B_{3\d}}{ \lvert \nabla u_\d(x_1,x_2) - \nabla u_\d(x_1+\frac{\d}{2},x_2) \lvert^2 \:dx} + \int_{H_\delta^- \setminus B_{3\d}}{ \lvert \nabla u_\d(x_1,x_2) - \nabla u_\d(x_1+\frac{\d}{2},x_2) \lvert^2 \:dx} , \\
&=:& J_\d^1 + J_\d^2.
\end{array}
$$
At first, for $x \in B_{3\d}$, one has $u_\d(x) = s_\d \chi_1(\frac{r}{\d}) u(x)$, and so:
\begin{equation}\label{eq.dudxiB3d}
 \frac{\partial u_\delta}{\partial x_i}(x_1,x_2) = s_\d \left( \frac{1}{\d}\frac{x_i}{r} \chi_1^\prime(\frac{r}{\d}) u(x) +  \frac{i\xi x_i}{2r^2} \chi_1(\frac{r}{\d}) (r^{i\xi} - r^{-i\xi}) \varphi(\theta)\right), \:\: i=1,2.
 \end{equation}
Now using Taylor's formula yields: 
$$ 
\begin{array}{>{\displaystyle}cc>{\displaystyle}l}
J_\d^1 & \leq & C\d^2 \left( \int_{H_\delta^- \cap B_{3\d}}{\int_0^1{ \left\lvert \frac{\partial^2 u_\delta }{\partial x_1^2} (x_1+t\frac{\d}{2},x_2)  \right\lvert^2 \:dt} \:dx} +  \int_{H_\delta^- \cap B_{3\d}}{\int_0^1{ \left\lvert \frac{\partial^2 u_\d }{\partial x_1\partial x_2}(x_1+t\frac{\d}{2},x_2)  \right\lvert^2 \:dt} \:dx} \right) , \\
&\leq&  C \d^2s_\d ^2 \int_{H_\delta^- \cap B_{3\d}}{\int_0^1{ \left( \frac{1}{\delta^4} + \frac{1}{\delta^2 r_t^2} + \frac{1}{r_t^4}\right) (\lvert \chi_1(\frac{r_t}{\d})\lvert^2 + \lvert \chi_1^\prime(\frac{r_t}{\d})\lvert^2 + \lvert \chi_1^{\prime\prime}(\frac{r_t}{\d})\lvert^2)   \:dt} \:dx},
\end{array}$$

where we have denoted by $(r_t, \theta_t)$ the polar representation of the point with Cartesian coordinates 
$(x_1+t\frac{\d}{2},x_2)$. Using that $\chi_1(\frac{r_t}{\d})$ vanishes for $r_t \leq \d$, it follows: 
$$ 
J_\d^1  \leq  C \d^2s_\delta^2 \int_{H_\delta^- \cap B_{3\d}}{\frac{1}{\d^4}\:dx},
$$
and so $J_\d^1$ converges to $0$ as $\d \to 0$, owing to the estimate (\ref{estim_se2}) on $s_\d$.

\medskip 
Let us now deal with the integral $J_\d^2$. Using the same calculation as above yields: 
$$ 
J_\d^2  \leq C\d^2 \left( \int_{H_\delta^- \setminus\overline{ B_{3\d}}Ê}{\int_0^1{ \left\lvert \frac{\partial^2 u_\delta }{\partial x_1^2} (x_1+t\frac{\d}{2},x_2)  \right\lvert^2 \:dt} \:dx} +  \int_{H_\delta^- \setminus \overline{B_{3\d}} }{\int_0^1{ \left\lvert \frac{\partial^2 u_\d }{\partial x_1\partial x_2}(x_1+t\frac{\d}{2},x_2)  \right\lvert^2 \:dt} \:dx} \right) , 
$$
and since for $x \in \Omega \setminus \overline{B_{3\d}}$, one has $u_\d(x) = s_\d u(x) \chi_2(r)$, it follows: 
\begin{equation}\label{eq.dudxiextB3d}
 \frac{\partial u_\d}{\partial x_i}(x) = s_\d \frac{i\xi x_i}{2r^2}(r^{i\xi}-r^{-i\xi}) \varphi(\theta) \chi_2(r) + \frac{x_i}{r}\chi_2^\prime(r) u(x), \text{ for }Êx \in \Omega \setminus \overline{B_{3\d}},
 \end{equation}
so that: 
$$ J_\d^2 \leq C \delta^2 s_\d ^2 \int_{B_{2\rho} \setminus \overline{B_{3\d}}}{\int_0^1 \left\lvert \frac{1}{r_t^2} \right\lvert^2 \:dx},$$
where, again, $(r_t,\theta_t)$ are the polar coordinates of $(x_1+t\frac{\d}{2},x_2)$.
We now remark that, by an elementary calculation: 
$$ r_t^2 \geq \frac{r^2}{2} - \frac{\delta^2}{2},$$
so that, switching to polar coordinates: 
$$ J_\d^2 \leq C\delta^2 s_\delta^2 \int_{3\d} ^{2\rho}{\frac{r dr}{(r^2-\delta^2)^2}} \leq C s_\delta^2,$$
whence $J_\d^2 \to 0$. This completes the proof of that fact that $I_\d^- \to 0$ as $\d \to 0$.
\medskip 

$\bullet$ The proof that $I_\d^+ \to 0$ is completely similar. 
\par\medskip

$\bullet$ \textit{Proof of the convergence $I_\d^L \to 0$:} 
Using a similar decomposition as in the case for $I_\d^-$, we get: 
$$ \begin{array}{>{\displaystyle}cc>{\displaystyle}l}
I_\d^L &=& \int_{L_\d \cap B_{3\d}}{ \lvert \nabla u_\d(x_1,x_2) - \nabla u_\d(0,x_2) \lvert^2 \:dx}  +  \int_{L_\d \setminus \overline{B_{3\d}}}{ \lvert \nabla u_\d(x_1,x_2) - \nabla u_\d(0,x_2) \lvert^2 \:dx}, \\
&=: & K_\d^1 + K_\d^2.
\end{array} $$
Using the expression (\ref{eq.dudxiB3d}) for the gradient of $u_\d$ inside $B_{3\d}$, it comes:
$$ \begin{array}{>{\displaystyle}cc>{\displaystyle}l}
K_\d^1 &=& \int_{L_\d \cap B_{3\d}}{ \left\lvert \int_{0}^1{\frac{\partial^2 u_\d}{\partial x_1^2}(t x_1,x_2) x_1 \:dt} \right \lvert^2 \:dx}  + \int_{L_\d \cap B_{3\d}}{ \left\lvert \int_{0}^1{\frac{\partial^2 u_\d}{\partial x_1 \partial x_2}(t x_1,x_2) x_1 \:dt} \right \lvert^2 \:dx} \\
&\leq&  C s_\d ^2 \int_{H_\delta^- \cap B_{3\d}}{\int_0^1{ \left( \frac{1}{\delta^4} + \frac{1}{\delta^2 r_t^2} + \frac{1}{r_t^4}\right) (\lvert \chi_1(\frac{r_t}{\d})\lvert^2 + \lvert \chi_1^\prime(\frac{r_t}{\d})\lvert^2 + \lvert \chi_1^{\prime\prime}(\frac{r_t}{\d})\lvert^2)  \lvert x_1 \lvert^2  \:dt} \:dx},
\end{array} $$
where we have now denoted by $(r_t,\theta_t)$ the polar coordinates of $(tx_1,x_2)$.
Since $\chi_1(\frac{r_t}{\delta})$, $\chi_1^\prime(\frac{r_t}{\delta})$ and $\chi_1^{\prime\prime}(\frac{r_t}{\delta})$ vanish identically for $r_t \leq \delta$, we obtain: 
$$K_\d^1 \leq Cs_\d^2 \int_{L_\d \cap B_{3\d}}{\frac{\lvert x_1\lvert^2}{\delta^4}\:dx},$$
and it follows as previously that $K_\d^1 \to 0$ as $\d \to 0$. Likewise, using (\ref{eq.dudxiextB3d}), we get: 
$$ \begin{array}{>{\displaystyle}cc>{\displaystyle}l}
K_\d^2 &=& \int_{L_\d \setminus \overline{B_{3\d}} }{ \left\lvert \int_{0}^1{\frac{\partial^2 u_\d}{\partial x_1^2}(t x_1,x_2) x_1 \:dt} \right \lvert^2 \:dx}  + \int_{L_\d \setminus \overline{B_{3\d}} }{ \left\lvert \int_{0}^1{\frac{\partial^2 u_\d}{\partial x_1 \partial x_2}(t x_1,x_2) x_1 \:dt} \right \lvert^2 \:dx} \\
&\leq&  C s_\d ^2 \int_{L_\d \setminus \overline{B_{3\d}} }{ \int_0^1{\frac{1}{\lvert r_t \lvert^4} \lvert x_1\lvert^2\:dt} \:dx},
\end{array}
$$
where $(r_t,\theta_t)$ are the polar coordinates of $(tx_1,x_2)$. We now use the fact that, for $x \in L_\d \setminus \overline{B_{3\d}}$ and $t \in (0,1)$, 
$$ \begin{array}{ccl}
r_t^2 &=& t^2x_1^2 + x_2^2 , \\ 
&=& r^2 + (t^2 - 1) x_1^2, \\
&\geq & r^2 - \frac{\delta^2}{4}. 
\end{array}$$
Hence, switching to polar coordinates, 
$$ \begin{array}{>{\displaystyle}cc>{\displaystyle}l}
K_\d^2 &\leq& Cs_\d^2  \int_{L_\d \setminus \overline{B_{3\d}}}{\frac{\lvert x_1\lvert^2}{(r^2 - \delta^2)^2}\:dx}, \\
&\leq& C \d^2 s_\d^2  \int_{L_\d \setminus \overline{B_{3\d}}}{\frac{1}{(r^2 - \delta^2)^2}\:dx}, \\
&\leq& C \d^2 s_\d^2  \int_{3\d}^{2\rho}{\frac{r}{(r^2 - \delta^2)^2}\:dr}, \\
&\leq& Cs_\d^2, 
\end{array}
$$
which completes the proof of the fact that $K_\d^2 \to 0$ as $\d \to 0$, and so that $I_\d^L \to 0$.\par
\medskip
Putting things together, we have proved that $\lvert\lvert u_\d - w_\d \lvert\lvert_{H^1_0(\Omega)}^2 =  I_\d^- + I_\d^+ + I_\d^L + o(1)$ 
converges to $0$ as $\d \to 0$, which is the expected conclusion.
\end{proof}


\noindent \textbf{Acknowledgements.} 
Hai Zhang was partially supported by Hong Kong RGC grant ECS 26301016 and 
startup fund R9355 from HKUST.
E. Bonnetier, C. Dapogny and F. Triki were partially supported by the AGIR-HOMONIM grant from 
Universit\'e Grenoble-Alpes, and by the Labex PERSYVAL-Lab (ANR-11-LABX-0025-01).
This project was conducted while E.B. was visiting the Institute
of Mathematics and its Applications at the University of Minnesota,
the hospitality and support of which is gratefully acknowledged.


\appendix
\section{The spectrum of an operator and the Weyl criterion}\label{sec.weyl}

For the reader's convenience, we recall in this appendix the Weyl criterion, one of the main tools used in the present article; 
see for instance \cite{ReedSimon}, Chap. VII or \cite{BirmanSolomjak} for a more complete presentation.
\medskip

Let $T: H \to H$ be a bounded self-adjoint operator on a Hilbert space $H$. 
As is well-known, the spectrum $\sigma(T)$ of $T$ is the set of real numbers $\lambda$ such that $(\lambda \Id - T)$ does not have a bounded inverse.
The \textit{discrete spectrum} $\sigma_{\text{\rm disc}}(T)$ of $T$ is the subset of the $\lambda \in \sigma(T)$ such that both the following conditions hold: 
\begin{enumerate}[(i)]
\item $\lambda$ is isolated in $\sigma(T)$, i.e. there exists $\varepsilon >0$ such that $\sigma(T) \cap (\lambda-\varepsilon, \lambda + \varepsilon) = \left\{ \lambda \right\}$, 
\item $\lambda$ is an eigenvalue of $T$ with finite multiplicity. 
\end{enumerate}
The complement of $\sigma_{\text{\rm disc}}(T)$ in $\sigma(T)$ is a closed set called the \textit{essential spectrum} of $T$
and is denoted by $\sigma_{\text{\rm ess}}(T)$.
\medskip

The Weyl criterion offers a convenient characterization of the spectrum and essential spectrum in terms of \textit{Weyl sequences}:

\begin{theorem}\label{th.weyl}
Let $T : H \to H$ be a bounded, self-adjoint operator on a Hilbert space $H$. Then, 
\begin{itemize}
\item A number $\lambda \in \R$ belongs to the spectrum $\sigma(T)$ if and only if there exists a sequence $u_n \in H$ such that: 
$$ \lvert\lvert u_n \lvert\lvert = 1  \text{ and } \lvert\lvert \lambda u_n - T u_n \lvert\lvert \xrightarrow{n\to \infty}Ê0.$$ 
Such a sequence is called a Weyl sequence for $T$ associated to the value $\lambda$. 
\item $\lambda \in \R$ belongs to the essential spectrum $\sigma_{\text{\rm ess}}(T)$ if and only if there exists a Weyl sequence $u_n$ for $\lambda$ such that $u_n \to 0$ weakly in $H$; such a sequence is called a singular Weyl sequence for $T$ and $\lambda$.
\end{itemize}
\end{theorem}


\end{document}